\documentclass[a4paper,11pt]{amsart}
\usepackage{amsmath,amssymb,amsfonts,amscd,hyperref,verbatim}
\usepackage{ytableau,array,enumitem}

\theoremstyle{plain}
\newtheorem{thm}{Theorem}[section]
\newtheorem{cor}[thm]{Corollary}
\newtheorem{lem}[thm]{Lemma}
\newtheorem{prop}[thm]{Proposition}
\theoremstyle{definition}
\newtheorem{defn}[thm]{Definition}
\theoremstyle{remark}
\newtheorem{rem}[thm]{Remark}
\newtheorem{eg}[thm]{Example}

\newtheorem*{eg*}{Example}

\newcommand{\ZZ}{\mathbb{Z}}
\newcommand{\QQ}{\mathbb{Q}}

\newcommand{\ZP}{\mathbb{Z}_{({p})}}
\DeclareMathOperator{\sgn}{sgn}
\newcommand{\FF}{\mathbb{F}}

\DeclareMathOperator{\Hom}{Hom}

\DeclareMathOperator{\Tor}{Tor}

\DeclareMathOperator{\Tab}{\mathcal{T}}
\DeclareMathOperator{\STab}{Std}                          

\newcommand{\s}{\mathfrak{s}}
\renewcommand{\t}{\mathfrak{t}}
\DeclareMathOperator{\h}{\mathfrak{h}}

\newcommand{\IT}{\t}   
\newcommand{\sym}[1]{\mathfrak{S}_{#1}}
\newcommand{\phl}{\mathsf{h}}
\newcommand{\dd}{\mathsf{d}}
\newcommand{\Shape}{\mathrm{Shape}}    
\newcommand{\lcm}{\mathrm{lcm}}
\newcommand{\C}{\mathcal{C}}
\newcommand{\m}{\mathbf{m}}

\newcommand{\FM}[1]{\textcolor{red}{#1}}
\newcommand{\KM}[1]{\textcolor{violet}{#1}}

\newcommand{\add}[2]{{#1} \oslash {#2}}

\title[Jantzen filtration, Young symmetrizers and seminormal basis]{Jantzen filtration of Weyl modules,\\product of Young symmetrizers and\\denominator of Young's seminormal basis}
\author{Ming Fang}
\address[M. Fang]{HLM, HCMS, Academy of Mathematics and Systems Science, Chinese Academy of Sciences, Beijing, 100190 -and- School of Mathematical Sciences, University of Chinese Academy of Sciences, Beijing, 100049, People's Republic of China.}
\email{fming@amss.ac.cn}

\author{Kay Jin Lim}
\address[K. J. Lim]{Division of Mathematical Sciences, Nanyang Technological University, SPMS-04-01, 21 Nanyang Link, Singapore 637371.}
\email{limkj@ntu.edu.sg}

\author{Kai Meng Tan}
\address[K. M. Tan]{Department of Mathematics, National University of Singapore, Block S17, 10 Lower Kent Ridge Road, Singapore 119076.}
\email{tankm@nus.edu.sg}

\subjclass[2010]{20G05, 20C30}
\thanks{The first author is supported by NSFC (No.\ 11688101, 11471315 and 11321101), while the second and third authors are supported by Singapore MOE Tier 2 AcRF MOE2015-T2-2-003.}
\keywords{Jantzen filtration, Young symmetrizer, Young's seminormal basis}

\date{July 2020}

\begin{document}

\begin{abstract}
Let $G$ be a connected reductive algebraic group over an algebraically closed field of characteristic $p>0$, $\Delta(\lambda)$ denote the Weyl module of $G$ of highest weight $\lambda$ and  $\iota_{\lambda,\mu}:\Delta(\lambda+\mu)\to \Delta(\lambda)\otimes\Delta(\mu)$ be the canonical $G$-morphism. We study the split condition for  $\iota_{\lambda,\mu}$ over $\ZZ_{(p)}$, and apply this as an approach to compare the Jantzen filtrations of the Weyl modules $\Delta(\lambda)$ and $\Delta(\lambda+\mu)$. In the case when $G$ is of type $A$, we show that the split condition is closely related to the product of certain Young symmetrizers and, under some mild conditions, is further characterized by the denominator of a certain Young's seminormal basis vector. We obtain explicit formulas for the split condition in some cases.
%
%
%
%
\end{abstract}

\maketitle

\section{Introduction}

Let $G_{\ZZ}$ be a split connected reductive group defined over $\ZZ$, and $G$ denote the algebraic group over an algebraically closed field $\FF$ of prime characteristic $p$ obtained from $G_{\ZZ}$.  The Jantzen filtration of a Weyl module of $G$, introduced in \cite{Jantzen77}, enjoys a rich structure, which led to many remarkable results (see, for example, \cite{Andersen83,Andersen87,Jantzen03,LyleMathas10,Shan12}), giving us a more complete understanding of the representation theory of $G$.

Despite these advances, how these filtrations for different Weyl modules are related remains to this day a very difficult open problem.  Conjectures concerning this open problem for simply connected semisimple algebraic groups include Jantzen's conjecture, stated in \cite{Andersen83}, that relates the Jantzen filtrations of two Weyl modules with adjacent highest weights, and Xi's conjecture \cite[Conjecture H]{Xi96} which implies a relation between the Jantzen filtrations of $\Delta(\lambda)$ and that of $\Delta(\lambda+p(p-1)\rho)$, where $\lambda$ is a $p$-restricted dominant integral weight and $\rho$ is the half sum of all positive roots.

The main results in this paper are Theorems \ref{thm:split-Jantzen-filtration}, \ref{thm:split-condition-main} and \ref{thm:denominator}.
We first concern ourselves with the relationship between the Jantzen filtrations of the Weyl modules $\Delta(\lambda)$ and $\Delta(\lambda + \mu)$ for two dominant integral weights $\lambda$ and $\mu$.  Let $\mathbb{Z}_{(p)}$ be the ring $\mathbb{Z}$ localized at the prime ideal $(p)$, and write $\Delta_{\mathbb{Z}_{(p)}}(\lambda)^i$ for the $i$-th term in the Jantzen filtration of the Weyl module $\Delta_{\mathbb{Z}_{(p)}}(\lambda)$ over $\mathbb{Z}_{(p)}$, and similarly for the weights $\mu$ and $\lambda+\mu$.  Our first main result (Theorem \ref{thm:split-Jantzen-filtration}) states that if the canonical $G$-morphism $\iota_{\lambda,\mu} : \Delta(\lambda+\mu) \to \Delta(\lambda) \otimes \Delta(\mu)$ admits a splitting map defined over $\mathbb{Z}_{(p)}$, then $\Delta_{\mathbb{Z}_{(p)}}(\lambda)^i$ may be naturally embedded into $\Delta_{\mathbb{Z}_{(p)}}(\lambda+\mu)^i$ as $\ZZ_{(p)}$-modules for all $i$. In particular, the induced $\FF$-linear map from $\Delta(\lambda)$ to $\Delta(\lambda+\mu)$ respects the Jantzen filtrations, see Remark \ref{rem:Jantzen}. 

The split condition for $\iota_{\lambda,\mu}$ over $\ZZ_{(p)}$ appears to be of independent interest. Indeed, Andersen communicated to us some necessary conditions for this splitting over $\FF$ (hence over $\ZZ_{(p)}$) by considering the restriction to $\mathrm{SL}_2$; see Proposition \ref{prop:andersen-criterion}. In the case when $G$ is the general linear group, the split condition for $\lambda=(m)$ and $\mu=(n)$ is obtained by Donkin \cite[\S4.8(12) Proposition]{Donkin98} (see also Proposition \ref{prop:Donkin-criterion}) and has played a crucial role in the determination of the global dimensions of Schur algebras \cite{T97} (see also \cite[\S4.8]{Donkin98}).  We note also that for type $A$, the split condition for $\lambda$ an arbitrary partition and $\mu=(1)$ may be approached using the theory of translation functors developed in \cite{BK}.

Having shown that the splitting of $\iota_{\lambda,\mu}$ over $\ZP$ plays a significant role in the comparison of the Jantzen filtrations of $\Delta(\lambda)$ and $\Delta(\lambda+\mu)$, we next seek to determine some necessary and sufficient conditions for this splitting for the case of the general linear groups.  Our second main result (Theorem \ref{thm:split-condition-main}) gives such a condition in terms of $\theta_{\lambda,\mu}$, which is the greatest common divisor of the coefficients of the product of certain Young symmetrizers in the integral group ring of symmetric groups.

Young's seminormal basis (see \cite{Mathasbook99,Murphy92} and the references therein), a substantial ingredient nowadays in the representation theory of symmetric groups, is by definition only a $\mathbb{Q}$-basis of the group algebra of symmetric groups. The denominators of these basis elements are not known in general, but seem to control certain parts of the modular representation theory (\cite{Mathasbook99,RHansen10,RHansen13}). Our third main result (Theorem \ref{thm:denominator}) relates further the above split condition to the mysterious denominator of a certain Young's seminormal basis element. To be precise, assume that the Young diagram of $\lambda+\mu$ can be obtained by putting the Young diagrams of $\lambda$ and $\mu$ side by side; then the split condition can be characterized in terms of the denominator $\dd_{\add{\IT^{\lambda}}{\IT^{\mu}}}$ of $f_{\add{\IT^{\lambda}}{\IT^{\mu}}}$, an element of Young's seminormal basis of the dual Specht module $S^{\QQ}_{\lambda+\mu}$. This particular denominator is not well studied yet, to our knowledge, and we leave further investigation in our future work.

We conclude the paper with an explicit computation of $\theta_{\lambda,\mu}$ for two cases: $\lambda = (1^n)$ and $\mu = (m)$, and $\lambda = (k,\ell)$ and $\mu = (m)$. In particular, the determination of $\theta_{(k,\ell),(m)}$ gives us a split condition for $\iota_{(k,\ell),(m)}$ over $\ZP$ which generalizes the aforementioned result of Donkin.

The paper is organized as follows: in the following section, we recall the necessary background, fix some notations which shall be used throughout and prove some preliminary results.  In Section \ref{sec:main results}, we state and prove our main results, namely Theorems \ref{thm:split-Jantzen-filtration}, \ref{thm:split-condition-main} and \ref{thm:denominator}.  We then conclude in Section \ref{sec:examples} with the explicit computations of the products of Young symmetrizers that lead to the closed formulas for $\theta_{(1^n),(m)}$ and $\theta_{(k,\ell),(m)}$.

\section{Preliminaries} \label{sec:preliminaries}
In this section, we recall the background theory that we require, fix all relevant notations and prove some preliminary results.  
Throughout this article, $\FF$ denotes an algebraically closed field of prime characteristic $p$, and $\ZP$ denotes the ring of integers localized at the prime ideal $(p)$.  We identify $\ZP$ with the subring of $\QQ$ consisting of all rational numbers with denominators not divisible by $p$, and note that $\FF$ is a natural $\ZP$-module.

We remark that our results in fact hold even when $\FF$ is not algebraically closed, but we assume $\FF$ to be algebraically closed here for the ease of presentation, which avoids the discussion of group schemes when defining the Weyl modules.

\subsection{Weyl modules and Jantzen filtration} \label{subsec:Weyl-module}

Following \cite[II, Chapter 1]{Jantzen03}, for a split connected reductive group $G_{\ZZ}$ defined over $\ZZ$, let $\mathrm{Dist}(G_{\ZZ})$ be its distribution algebra, and let $T_{\ZZ}$ be a split maximal torus contained in a split Borel subgroup $B_{\ZZ}$ of $G_{\ZZ}$. For a commutative ring $\mathcal{O}$ with $1$, set $G_{\mathcal{O}} = (G_{\ZZ})_{\mathcal{O}}$, $T_{\mathcal{O}} = (T_{\ZZ})_{\mathcal{O}}$, and  $\mathrm{Dist}(G_{\mathcal{O}}) = \mathrm{Dist}(G_{\ZZ})\otimes_{\ZZ} {\mathcal{O}}$. By \cite[II, Section 1.12(2)]{Jantzen03}, $\mathrm{Dist}(G_{\mathcal{O}})$ admits a triangle decomposition
\[
\mathrm{Dist}(G_{\mathcal{O}}) \cong \mathrm{Dist}(G_{\mathcal{O}})^+\otimes_{\mathcal{O}} \mathrm{Dist}(G_{\mathcal{O}})^0 \otimes_{\mathcal{O}} \mathrm{Dist}(G_{\mathcal{O}})^- .
\]
Since rational $G_{\mathcal{O}}$-modules are identified with locally finite $\mathrm{Dist}(G_{\mathcal{O}})\textrm{-}T_{\mathcal{O}}$-modules---see \cite[II, Section 1.20]{Jantzen03}  for the technical notions and assumptions involved---we use this distribution algebra approach with some care about the action of $T_{\mathcal{O}}$ to describe these $G_{\mathcal{O}}$-modules.

For a dominant integral weight $\lambda$, let $\Delta_{\QQ}(\lambda)$ denote the finite-dimensional irreducible $G_{\QQ}$-module with highest weight $\lambda$.
Fix a nonzero highest weight vector $\eta_{\lambda} \in \Delta_{\QQ}(\lambda)$, and let $\Delta_{\ZZ}(\lambda) = \mathrm{Dist}(G_{\ZZ})\eta_{\lambda}$ and $\Delta_{\mathcal{O}}(\lambda) = \Delta_{\ZZ}(\lambda) \otimes_{\ZZ} \mathcal{O}$. Note that different choices of $\eta_\lambda$ result in the same $G_{\mathcal{O}}$-module $\Delta_{\mathcal{O}}(\lambda)$ (up to isomorphism).

Denote by $\tau$ the Cartan involution on $\mathrm{Dist}(G_{\mathcal{O}})$; thus $\tau$ fixes $\mathrm{Dist}(G_{\mathcal{O}})^0$ pointwise and interchanges $\mathrm{Dist}(G_{\mathcal{O}})^+$ and $\mathrm{Dist}(G_{\mathcal{O}})^-$. A symmetric bilinear form $[-,-]$ on a $\mathrm{Dist}(G_{\mathcal{O}})$-module $M$ is said to be {\em contravariant} if $[zu,v] = [u, \tau(z)v]$ for all $u,v \in M$ and $z \in \mathrm{Dist}(G_{\mathcal{O}})$.

It is well known that there exists a contravariant symmetric bilinear form $(-,-)$ on $\Delta_{\mathcal{O}}(\lambda)$ such that
$
(\eta_\lambda, \eta_\lambda) =1$ and   $(tu,v)= (u,tv)
$
for all $t\in T_{\mathcal{O}}$ and $u,v \in \Delta_{\mathcal{O}}(\lambda)$, which is non-degenerate when $\mathcal{O}$ is a field of characteristic $0$.  In addition, all other contravariant symmetric bilinear forms on $\Delta_{\mathcal{O}}(\lambda)$ are scalar multiples of $(-,-)$.

When $\mathcal{O} = \ZP$, we write $c_{\lambda}$ for $(-,-)$, and define
\[
   \Delta_{\ZZ_{(p)}}(\lambda)^i=
    \{x \in \Delta_{\ZP}(\lambda) \mid c_\lambda(x, y)\in p^i\ZZ_{(p)}, \,
    \forall y\in \Delta_{\ZZ_{(p)}}(\lambda)\}.
 \]
 Then
 $$
 \Delta_{\ZZ_{(p)}}(\lambda) \supseteq \Delta_{\ZZ_{(p)}}(\lambda)^1 \supseteq \Delta_{\ZZ_{(p)}}(\lambda)^2 \supseteq \dotsb
 $$
 is the {\em Jantzen filtration} of $\Delta_{\ZZ_{(p)}}(\lambda)$.  Writing $\Delta(\lambda)^i$ for the image of $\Delta_{\ZZ_{(p)}}(\lambda)^i$ in the Weyl module $\Delta(\lambda):= \Delta_{\FF}(\lambda)$, we have the corresponding Jantzen filtration
 $$
 \Delta(\lambda) \supseteq \Delta(\lambda)^1 \supseteq \Delta(\lambda)^2 \supseteq \dotsb
 $$
 of $\Delta(\lambda)$, see \cite[II, Chapter 8]{Jantzen03}.

\subsection{Symmetric groups}
Denote the group of bijections on a nonempty set $X$ by $\sym{X}$, and further write $\sym{n}$ for $\sym{\{1,\dotsc, n\}}$.  We view elements of such a group as functions, so that we compose these elements from right to left. When $Y$ is a nonempty subset of $X$, we view $\sym{Y}$ as a subgroup of $\sym{X}$ by identifying an element of $\sym{Y}$ with its extension that sends $x$ to $x$ for all $x \in X \setminus Y$.

As usual, $\ZZ^+$ denotes the set of all positive integers. Let $X \subseteq \ZZ^+$ and $k \in \ZZ^+$.  Define $X^{+k} := \{ x+k \mid x \in X \}$, and for any function $\sigma : X \to X$, write $\sigma^{+k} : X^{+k} \to X^{+k}$ for the function such that $\sigma^{+k}(x+k) = \sigma(x) + k$ for all $x \in X$.  Then $\sigma \mapsto \sigma^{+k}$ is a group isomorphism from $\sym{X}$ to $\sym{X^{+k}}$, and this extends further to give an isomorphism $\ZZ\sym{X} \to \ZZ\sym{X^{+k}}$.  If $R \subseteq \ZZ\sym{X}$, we write $R^{+k}$ for $\{ r^{+k} \mid r\in R \}$.  In particular, $\sym{X}^{+k} = \sym{X^{+k}}$.

For a subset $S$ of $\sym{n}$, define $\{S\}, [S] \in \ZZ\sym{n}$ by
$$
\{S \} := \sum_{\sigma \in S} \sigma, \qquad [S] := \sum_{\sigma \in S} \sgn(\sigma) \sigma,
$$
where $\sgn(\sigma) \in \{\pm 1 \}$ is the usual signature of $\sigma$. 

\subsection{Partitions and Young tableaux}
\label{subsec:combinatorics-Young-tableaux}
Let $n$ be a natural number. A partition $\lambda$ of $n$, denoted $\lambda \vdash n$, is a non-increasing sequence  $\lambda=(\lambda_1,\lambda_2,\ldots)$ of non-negative integers such that $\lambda_1+\lambda_2+\cdots = n$. The dominance order $\unlhd$ on all partitions of $n$ is given by: $\lambda \unlhd \mu $ if and only if $\lambda_1+ \cdots+\lambda_k\leq \mu_1+\cdots+\mu_k$ for all $k\in \ZZ^+$.

The Young diagram of $\lambda$ is defined to be the set $ [\lambda] = \{(a,b) \in (\mathbb{Z}^+)^2 \mid b\leq \lambda_a\};$ and we call its elements the {\em nodes} of $\lambda$. The conjugate of $\lambda$ is the partition $\lambda'$ with $\lambda'_i=|\{j\mid (j,i)\in [\lambda]\}|$ for all $i$. We depict $[\lambda]$ as an array of left-justified boxes in which the $i$-th row comprises exactly $\lambda_i$ boxes, with each box representing a node of $\lambda$.

A $\lambda$-tableau is a bijective map $\s: [\lambda]\to \{1,\ldots, n\}$, and $\lambda$ is said to be the shape of $\s$, denoted by $\Shape(\s)$.  We identify $\s$ with the pictorial depiction of the Young diagram $[\lambda]$ in which each box in $[\lambda]$ is filled with $\{1,2,\dotsc, n\}$ so that each integer appears exactly once.  For $1 \leq r \leq n$, the residue $\mathrm{res}_{\s}(r)$ is equal to $j-i$ if $\s(i,j) = r$.  Denote the set of all $\lambda$-tableaux by $\Tab(\lambda)$.

A $\lambda$-tableau $\s$ is said to be \emph{standard} if the entries in $\s$ are increasing along each row and down each column in the Young diagram. Let $\STab(\lambda)$ be the set of all standard $\lambda$-tableaux. Let $\s \in \STab(\lambda)$ and $1\leq r\leq n$.  Since $\s$ is standard, $\s^{-1}(\{1,\dotsc, r \})$ is the Young diagram of a partition, and we define the subtableau  $\s{\downarrow_r}$ of $\s$ to be the restriction of $\s$ to this subdomain.  Pictorially, $\s{\downarrow_r}$ consists precisely of those boxes in $[\lambda]$ which are filled with $1,\dotsc,r$ in $\s$. The dominance order $\unlhd$ on $\STab(\lambda)$ is given by {$\s\unlhd\ \t$ if and only if, for each $1\leq r\leq n$, we have \[\Shape(\s{\downarrow_r})\unlhd  \Shape(\t{\downarrow_r}).\]

Let $\IT^\lambda$ be the $\lambda$-tableau such that $\IT^{\lambda} (a,b)= \lambda_1+\cdots+\lambda_{a-1}+b$. Similarly, let
$\IT_\lambda$ be the $\lambda$-tableau such that $\IT_{\lambda}(a,b) = a+\lambda'_1+\cdots+\lambda'_{b-1}$. For example,
$$
\IT^{(2,2,1)} =
\raisebox{3mm}{
\ytableausetup{boxsize=1em, mathmode}
\begin{ytableau}
\scriptstyle 1 & \scriptstyle 2   \\
\scriptstyle 3 & \scriptstyle 4  \\
\scriptstyle 5
\end{ytableau}}\ , \qquad
\IT_{(2,2,1)} =
\raisebox{3mm}{
\ytableausetup{mathmode}
\begin{ytableau}
\scriptstyle 1 & \scriptstyle 4  \\
\scriptstyle 2 & \scriptstyle 5  \\
\scriptstyle 3
\end{ytableau}}\ .
$$
It is well known that, with respect to $\unlhd$,  $\IT^\lambda$ and $\IT_\lambda$ are the largest and smallest respectively in $\STab(\lambda)$.  

Post-composition of $\lambda$-tableaux by elements of $\sym{n}$ gives a well-defined, faithful and transitive left action of $\sym{n}$ on $\Tab(\lambda)$, i.e.\ $\sigma \cdot \s = \sigma \circ \s$ for $\sigma \in \sym{n}$ and $\s\in \Tab(\lambda)$.  For a $\lambda$-tableau $\s$, let $d(\s)$ be the element in $\sym{n}$ such that $\s= d(\s) \cdot \IT^\lambda$, or equivalently $d(\s) =  \s \circ (\IT^{\lambda})^{-1}$. Furthermore, we write
$$\sigma_{\lambda} := d(\IT_{\lambda}).$$
We denote by $R_{\s}$ and $C_{\s}$ the row and column stabilizers of $\s$, respectively.  The associated Young symmetrizer $Y_{\s} \in \ZZ\sym{n}$ is defined as \[Y_{\s} := \{ R_{\s} \} [C_{\s}].\] It is well known that $Y_{\s}^2 = \phl^{\lambda}Y_{\s}$, where $\phl^{\lambda} := \frac{n!}{|\STab(\lambda)|} \in \ZZ^+$, and that if $\t = \sigma \cdot \s$, where $\sigma \in \sym{n}$, then $R_{\t} = \sigma R_{\s} \sigma^{-1}$ and $C_{\t} = \sigma C_{\s} \sigma^{-1}$, and so $Y_{\t} = \sigma Y_{\s} \sigma^{-1}$.

For a $\lambda$-tableau $\s$ and $k \in \ZZ^+$, define $\s^{+k} : [\lambda] \to \ZZ^+$ by $\s^{+k}(i,j) = \s(i,j) + k$ for $(i,j) \in [\lambda]$.  We view $\s^{+k}$ as a filling of the boxes in $[\lambda]$ by the numbers $k+1,\dotsc, k+n$.  We may thus speak of row and column stabilizers of $\s^{+k}$ too, which are subgroups of $\sym{n}^{+k}$.  Note that $R_{\s}^{+k} = R_{\s^{+k}}$ and $C_{\s}^{+k} = C_{\s^{+k}}$.

Let $\lambda \vdash n$ and $\mu \vdash m$, and let $\s \in \Tab(\lambda)$ and $\t \in \Tab(\mu)$.  We have $\lambda + \mu = (\lambda_1 + \mu_1, \lambda_2 + \mu_2,\dotsc) \vdash n+m$, and we now define a $(\lambda+\mu)$-tableau $\add{\s}{\t}$, which has the properties that $C_{\add{\s}{\t}} = C_{\s}C_{\t^{+n}}$ and $R_{\add{\s}{\t}} \supseteq R_{\s}R_{\t^{+n}}$.  To obtain $\add{\s}{\t}$, we insert the columns of $\t^{+n}$ into $\s$ successively, starting from the leftmost column and working towards the rightmost, such that a column of $\t^{+n}$ is inserted between two adjacent columns, the left of which is at least as long as the column to be inserted,  while the right of which is strictly shorter.  We illustrate this with the following example:

\begin{eg}
Let
 \begin{gather*}
 \ytableausetup{boxsize=1em,textmode,centertableaux,notabloids}
 \s=
 \begin{ytableau}
 1 & 3 & 6\\
 2 & 4 \\
 5
 \end{ytableau}, \qquad
 \t=
  \begin{ytableau}
   1 & 2 & 4 & 6 \\
   3 & 7 & 8\\
   5
 \end{ytableau}.
 \end{gather*}
 Then
 $$
 \t^{+6}=
 \begin{ytableau}
 7 & 8 & 10 & 12\\
 9 & 13 & 14 \\
 11
 \end{ytableau}, \qquad \text{and} \qquad
 \add{\s}{\t}=
 \begin{ytableau}
 1 & *(yellow)7 & 3 & *(yellow)8 & *(yellow)10 & 6 & *(yellow)12\\
 2 & *(yellow)9 & 4 & *(yellow)13 & *(yellow)14 \\
 5 & *(yellow)11
 \end{ytableau}.
$$
Thus
\begin{align*}
R_{\add{\s}{\t}} &= \sym{\{1,3,6,7,8,10,12\}} \sym{\{2,4,9,13,14\}}\sym{\{5,11\}}; \\
C_{\add{\s}{\t}} &=\sym{\{1,2,5\}}\sym{\{3,4\}}\sym{\{7,9,11\}}\sym{\{8,13\}}\sym{\{10,14\}}.
\end{align*}
\end{eg}

\subsection{Dual Specht modules}  \label{subsec:Specht-and-Young-basis}
Let $\lambda$ be a partition of $n$. We briefly review the construction of the permutation module $M^\lambda_{\ZZ}$ using $\lambda$-tabloids \cite[Chapter 7]{Fulton97}.  Two $\lambda$-tableaux $\s$ and $\t$ are \emph{row equivalent} if $\t = \sigma \cdot \s$ for some $\sigma\in R_{\s}$, and a $\lambda$-tabloid is a row equivalence class of $\lambda$-tableaux, which we usually write as $\{\t\}$ and depict, for example, as follows:
 \[
 \ytableausetup{boxsize=1em,tabloids}
 \ytableaushort{1654,72,3} =
 \ytableaushort{1456,27,3}
 \]
The left action of $\sym{n}$ on $\Tab(\lambda)$ induces an action on the set of $\lambda$-tabloids, i.e.\ $\sigma \cdot \{ \t \} = \{ \sigma \cdot \t\}$ for $\sigma \in \sym{n}$ and $\t \in \Tab(\lambda)$, and $M_{\ZZ}^{\lambda}$ is the associated permutation representation of this action over $\ZZ$. The {\em integral dual Specht module} $S^{\ZZ}_\lambda$ is then defined to be the quotient of
 $M^\lambda_{\ZZ}$ by the \emph{Garnir relations} \cite[\S7.4 Exercise 14]{Fulton97}: let $\{\t\}$ be a $\lambda$-tabloid, and $X$ be any $k$ elements in its $(i+1)$-th row of $\t$, then
 \[
 \{\t\}\equiv (-1)^k \sum \{\s\}
 \]
where the sum runs over all $\lambda$-tableaux $\s$ obtained from $\t$ by interchanging $X$ with $k$ elements in the $i$-th row, maintaining the orders of the two sets. (Readers are cautioned to the misprint of sign in \cite[\S7.4 Exercise 14]{Fulton97}.) 

For each $\s \in \Tab(\lambda)$, let $e_{\s}$ denote the image of $\{\s\}$ in $S_\lambda^{\ZZ}$ under the quotient map.  Since the action of $\sym{n}$ is transitive on $\Tab(\lambda)$, it is also transitive on $\{ e_{\s} \mid \s \in \Tab(\lambda) \}$, so that $S^{\ZZ}_{\lambda}$ is a cyclic $\ZZ\sym{n}$-module, generated by any $e_{\s}$.  It is well known that $\{ e_{\s} \mid \s \in \STab(\lambda) \}$ is a $\ZZ$-basis for $S^{\ZZ}_{\lambda}$,  called the \emph{standard basis}.
Furthermore, the $\ZZ\sym{n}$-morphism $\varphi_{\s}^{\ZZ}$ defined by $e_{\s} \mapsto Y_{\s}$ gives an isomorphism $S^{\ZZ}_{\lambda} \cong \ZZ\sym{n}Y_{\s}$.

Given a commutative ring $\mathcal{O}$ with $1$, define $S^\mathcal{O}_{\lambda} := \mathcal{O} \otimes_{\ZZ} S_{\lambda}^{\ZZ}$.  The above statements about $S^{\ZZ}_{\lambda}$ behave well under base change, so that analogous statements hold when $\ZZ$ is replaced by $\mathcal{O}$. In particular, for each $\s \in \STab(\lambda)$, we have an $\mathcal{O}\sym{n}$-isomorphism $\varphi_{\s}^{\mathcal{O}} : S^{\mathcal{O}}_{\lambda} \to \mathcal{O}\sym{n}Y_{\s}$ sending $e_{\s} \to Y_{\s}$.  Thus, from the following commutative diagram
$$
\begin{CD}
S^{\mathcal{O}}_{\lambda} @=  \mathcal{O} \otimes_{\ZZ} S^{\ZZ}_{\lambda} @> \mathrm{id}_{\mathcal{O}} \otimes \varphi_{\s}^{\ZZ} > \cong> \mathcal{O} \otimes_{\ZZ} \ZZ\sym{n}Y_{\s} \\
@V\varphi_{\s}^{\mathcal{O}}V \cong V     @.  @VV\mathrm{id}_{\mathcal{O}} \otimes \mathrm{inc}V \\
\mathcal{O} \sym{n} Y_\s @>\hookrightarrow>> \mathcal{O} \sym{n} @>\cong>> \mathcal{O} \otimes_{\ZZ} \ZZ\sym{n}
\end{CD}
$$
we conclude that the right vertical map $\mathrm{id}_{\mathcal{O}} \otimes \mathrm{inc} : \mathcal{O} \otimes_{\ZZ} \ZZ\sym{n}Y_{\s} \to \mathcal{O} \otimes_{\ZZ} \ZZ\sym{n}$, where $\mathrm{inc}$ is the inclusion map, is injective.  Hence, by considering the long exact sequence induced by the Tor functor on the short exact sequence $0 \to \ZZ{\sym{n}}Y_{\s} \to \ZZ\sym{n} \to \ZZ\sym{n}/\ZZ\sym{n}Y_{\s} \to 0$, we get $\Tor^{\ZZ}_1(\mathcal{O},\ZZ\sym{n}/\ZZ\sym{n}Y_{\s}) = 0$. This is true for all commutative ring $\mathcal{O}$ with $1$, so that in particular $\ZZ\sym{n}/\ZZ\sym{n}Y_{\s}$ is torsion-free, and hence free, as a $\ZZ$-module.  Consequently, $\ZZ\sym{n}Y_{\s}$ is a $\ZZ$-summand of $\ZZ\sym{n}$.

The set $\{ S^{\QQ}_{\lambda} \mid \lambda \vdash n \}$ is a complete set of pairwise non-isomorphic irreducible $\QQ\sym{n}$-modules.  In particular, the dimension of $S_{\lambda}^{\QQ}$, $|\STab(\lambda)|$, divides $n!$, the order of $\sym{n}$, so that indeed we have $\phl^{\lambda} = \frac{n!}{|\STab(\lambda)|} \in \ZZ^+$, as claimed in the Subsection \ref{subsec:combinatorics-Young-tableaux}.


\subsection{Young's seminormal basis}

Following \cite[Section 4]{Murphy92}, but considering the left $\sym{n}$-action (where composition of elements of $\sym{n}$ are from right to left) instead and taking the classical limit $q\rightarrow 1$,  we have the following constructions and facts.

For $2\leq k\leq n$, define the $k$-th Jucys-Murphy element $L_k:=(1,k)+\cdots+(k-1,k)\in \ZZ\sym{n}$, and let $R(k)$ be the set $\{i\in \ZZ\mid -k<i<k\}$ if $k\geq 4$, and
 $\{i\in \ZZ\mid -k<i<k, i\neq 0\}$ otherwise. For each $\lambda \vdash n$ and $\s \in \STab(\lambda)$, define, as the Jucys-Murphy elements pairwise commute,
 \[
 E_{\s} := \prod_{k=2}^n \prod_{\mbox{  }m\in R(k)\backslash \{\mathrm{res}_{\s}(k)\}} \dfrac{(L_k-m)}{\mathrm{res}_{\s}(k)-m} \in \QQ\sym{n}.
 \]
Then $\{ E_{\s} \mid \s \in \STab(\lambda),\ \lambda \vdash n \}$ is a complete set of pairwise orthogonal idempotents of $\QQ\sym{n}$ \cite[p.505, last paragraph]{Murphy92}.
 	
\begin{defn}
Let $\lambda \vdash n$.
\begin{enumerate}
  \item  For each $\t\in\STab(\lambda)$, define \[\gamma_{\t} := \prod_{(i,j) \in [\lambda]} \; \prod_{(k,\ell) \in \Gamma_{\t}(i,j)} \frac{(j-i) - (\ell-k) + 1}{(j-i)-(\ell-k)},\] where, for each $(i,j)\in[\lambda]$, \[\Gamma_{\t}(i,j) = \{ (k,\ell) \in [\lambda] \mid \ell<j,\ \t(k,\ell) < \t(i,j),\ \t(k',\ell) > \t(i,j)\	\forall k'>k \}\]
      (see \cite[Section 6]{Murphy92}).
  \item  For any $\s,\t\in\STab(\lambda)$, define $$f_{\s,\t}:= E_{\s}\, d(\s) \{ R_{\IT^{\lambda}} \} d(\t)^{-1} E_{\t} \in \QQ\sym{n},$$ and $f_{\s} := \gamma_{\IT^{\lambda'}} f_{\s,\IT_{\lambda}}$.
\end{enumerate}
\end{defn}

From now on, for each $1\leq i < n$, denote the basic transposition $(i,i+1) \in \sym{n}$ by $s_i$. 

 \begin{thm} \label{thm:basics-on-young-basis} \hfill
 \begin{enumerate}
   \item For $\lambda \vdash n$, we have $\gamma_{\IT^{\lambda'}}\gamma_{\IT_{\lambda}} = \phl^{\lambda}\, (= \frac{n!}{|\STab(\lambda)|})$.
   \item The group algebra $\QQ\sym{n}$ has a $\QQ$-basis $\{f_{\s,\t}\mid \lambda\vdash n,\
   \s,\t\in \STab(\lambda)\}$, called the {\em Young's seminormal basis}.
   \item For $\lambda,\mu\vdash n$, $\s,\t\in \STab(\lambda)$ and
   $\mathfrak{u},\mathfrak{v}\in \STab(\mu)$, we have 
   $f_{\s,\t}f_{\mathfrak{u},\mathfrak{v}} = \delta_{\t\mathfrak{u}} \gamma_{\t} f_{\s,\mathfrak{v}}$, where $\delta_{\t\mathfrak{u}}$ is the Kronecker delta.
   \item 
   For $\lambda \vdash n$ and $\s,\t \in \STab(\lambda)$, we have
   $$
   s_i f_{\s,\t} =
   \begin{cases}
   f_{\s,\t}, &\text{if $s_i\in R_{\s}$}; \\
   -f_{\s,\t}, &\text{if $s_i\in C_{\s}$}; \\
   a_i f_{\s,\t} + f_{s_i \cdot \s, \t}, &\text{if $s_i \cdot \s \in \STab(\lambda)$ and $\s \rhd\, s_i \cdot \s$}; \\
      a_i f_{\s,\t} + (1-a_i^2) f_{s_i \cdot \s, \t}, &\text{if $s_i \cdot \s \in \STab(\lambda)$ and $\s \lhd\, s_i \cdot \s$},
   \end{cases}
   $$
   where $a_i = (\mathrm{res}_{\s}(i+1) - \mathrm{res}_{\s}(i))^{-1}$.

    \item For $\lambda\vdash n$, recall that $\sigma_{\lambda} = d(\IT_{\lambda})$.  We have 
   \[
   Y_{\IT^{\lambda}} = f_{\IT^\lambda}\sigma_\lambda,
   \qquad \text{and} \qquad Y_{\IT_{\lambda}} = \sigma_\lambda f_{\IT^\lambda}.
   \]

   \item We have that $\{f_{\s}\mid \s\in \STab(\lambda)\}$
   is a $\QQ$-basis of $\QQ\sym{n}f_{\IT^{\lambda}} = \QQ\sym{n}Y_{\IT^{\lambda}} \sigma_{\lambda}^{-1}$, which is another  realisation of the dual Specht module $S^{\QQ}_\lambda$.  This basis is known as the {\em Young's seminormal basis of $S^{\QQ}_{\lambda}$}.
 \end{enumerate}
 \end{thm}
\begin{proof} Part (1) follows from the penultimate displayed equation on \cite[p.507]{Murphy92}.

For parts (2)--(4), we first prove that $E_{\s}x_{\s\t}E_{\t}$ and $\zeta_{\s\t}(= E_{\s} \xi_{\s\t} E_{\t})$ in \cite{Murphy92} are equal (we refer the reader to \cite{Murphy92} for the definitions of $x_{\s\t}$ and $\xi_{\s\t}$). By \cite[Theorem 4.5]{Murphy92}, we see that $\xi_{\s\t} = x_{\s\t} + h_{\s\t}$ where \[
h_{\s\t}\in \check{\mathcal{H}}^\lambda:=\mathrm{span}\{ x_{\mathfrak{u} \mathfrak{v}} \mid \mathfrak{u},\mathfrak{v} \in \STab(\mu),\ \mu \vdash n,\ \mu \rhd \lambda \}.\]
By \cite[3.20]{Mathasbook99} (note that $m_{\s\t}$ in \cite{Mathasbook99} is $x_{\s\t}$ in \cite{Murphy92}), $\{ \xi_{\s\t} \mid \s,\t \in \STab(\lambda),\ \lambda \vdash n \}$ is another basis for $\mathcal{H}$, and that, in fact,
\[\mathrm{span}\{ \xi_{\mathfrak{u} \mathfrak{v}} \mid \mathfrak{u},\mathfrak{v} \in \STab(\mu),\ \mu \vdash n,\ \mu \rhd \lambda \}=\check{\mathcal{H}}^\lambda\ni h_{\s\t}.\]
Now, $\xi_{\mathfrak{u}\mathfrak{v}} E_{\t} =0$ for all $ \mathfrak{u},\mathfrak{v} \in \STab(\mu)$ for some $\mu \vdash n$ with $\mu \rhd \lambda$ by \cite[(5.5)]{Murphy92}.
Thus, $\xi_{\s\t} E_{\t} = (x_{\s\t} + h_{\s\t})E_{\t} = x_{\s\t}E_{\t}$, and so $\zeta_{\s\t} = E_{\s}x_{\s\t}E_{\t}$.

Note that \cite{Murphy92} uses the convention of composing the elements of the symmetric group from left to right, opposite of that used here, so that our symmetric group algebra is the opposite ring of the Hecke algebra of \cite{Murphy92} at the limit $q \to 1$.  As such, our $f_{\t,\s}$ is precisely $E_{\s}x_{\s\t}E_{\t}=\zeta_{\s,\t}$ at the limit $q=1$, and thus parts (2) and (3) follow from p.505 to p.506, last displayed equation on p.510 of \cite{Murphy92} respectively. Part (4) also follows from \cite[Theorem 6.4]{Murphy92} in the same way, once the incorrect formula given there is corrected. The correct formula should be:
$$
\zeta_{us} T_v = \begin{cases}
-\frac{1}{[-h]_q} \zeta_{us}, &\text{if } |h| = 1; \\
-\frac{1}{[h]_q} \zeta_{us} + \zeta_{ut}, & \text{if } h > 1; \\
-\frac{1}{[h]_q} \zeta_{us} + \frac{q[h+1]_q[h-1]_q}{[h]_q^2} \zeta_{ut}, &\text{if } h < -1.
\end{cases}
$$
This mistake is rendered by another error in Lemma 6.2 of \cite{Murphy92}, which the formula depends on (the author erroneously attributed to Lemma 6.1 instead), where for the first displayed equation to hold, one needs to define $h$ to be $a-b$ instead of $b-a$.  (There is another minor error, inconsequential to the proof of Theorem 6.4, in the second displayed equation in Lemma 6.2 too).  The readers are welcome to verify that the correct results are as claimed above.

For part (5), using the penultimate display equation on page 511 of \cite{Murphy92} to our context,
$$
[C_{\IT_{\lambda}}] \sigma_{\lambda} \{ R_{\IT^{\lambda}} \} =
\gamma_{\IT^{\lambda'}}f_{\IT_{\lambda},\IT^{\lambda}}.$$
Applying the anti-automorphism of $\QQ\sym{n}$ defined by $\sum_{\sigma\in \sym{n}} a_{\sigma} \sigma \mapsto \sum_{\sigma \in \sym{n}} a_{\sigma} \sigma^{-1}$ to the above equation, since $E_{\s}$ is invariant under the anti-automorphism, we get
$$
f_{\IT^{\lambda}}=\gamma_{\IT^{\lambda'}} f_{\IT^{\lambda},\IT_{\lambda}}=\{ R_{\IT^{\lambda}} \}\sigma_{\lambda}^{-1} [C_{\IT_{\lambda}}]=\{ R_{\IT^{\lambda}} \} [C_{\IT^{\lambda}}]\sigma_{\lambda}^{-1}=Y_{\IT^\lambda}\sigma_\lambda^{-1},
$$
and, similarly, $f_{\IT^{\lambda}}=\sigma_\lambda^{-1}\{R_{\IT_{\lambda}}\}[C_{\IT_{\lambda}}]=\sigma_\lambda^{-1}Y_{\IT_{\lambda}}$.

Part (6) follows immediately from parts (2), (3) and (5).	
\end{proof}

\begin{rem}
Any scalar multiple of the $f_{\s}$'s would of course also give a $\QQ$-basis of $\QQ\sym{n}Y_{\IT^\lambda} \sigma_{\lambda}^{-1}$.  We choose the scaling for the $f_{\s}$'s so that $f_{\IT^{\lambda}} = Y_{\IT^{\lambda}} \sigma_{\lambda}^{-1}$ (cf. Theorem \ref{thm:basics-on-young-basis}(5)).
\end{rem}

The following are the results about Young's seminormal basis that we require in this paper:

\begin{prop}	\label{prop:products-of-young-basis-of-different-size}
Let $\lambda \vdash n$ and $\mu \vdash m$, and let $\s,\t \in \STab(\lambda)$ and $\mathfrak{u}, \mathfrak{v} \in \STab(\mu)$.
\begin{enumerate}
\item If $m\leq n$, then $f_{\mathfrak{u},\mathfrak{v}}f_{\s,\t} = 0$ unless $\s {\downarrow_m} = \mathfrak{v}$, in which case,
$$f_{\mathfrak{u},\mathfrak{v}}f_{\s,\t} = \gamma_{\mathfrak{v}} f_{\s',\t},$$ where $\s' \in \STab(\lambda)$ is obtained from $\s$ by replacing its subtableau $\mathfrak{s}{\downarrow_m}$ by $\mathfrak{u}$.
\item Suppose that the last column of $[\lambda]$ is no shorter than the first column of $[\mu]$, and let $\sigma \in \sym{m}$.  If $\sigma f_{\mathfrak{u},\mathfrak{v}} =\sum_{\mathfrak{w} \in \STab(\mu)} a_{\mathfrak{w}}f_{\mathfrak{w},\mathfrak{v}}$  ($a_{\mathfrak{w}}\in \QQ$), then
\[
	\sigma^{+n}f_{\add{\s}{\mathfrak{u}},\mathfrak{q}}=\sum a_{\mathfrak{w}} f_{\add{\s}{\mathfrak{w}},\mathfrak{q}}
\]
for any $\mathfrak{q} \in \STab(\lambda+\mu)$.
\item We have
\[
	d(\s) f_{\IT^{\lambda}, \t} = f_{\s,\t} + \sum_{\substack{\mathfrak{r} \in \STab(\lambda) \\ \mathfrak{r} \rhd \s}} a_{\mathfrak{r}} f_{\mathfrak{r}, \t}  \qquad ( a_{\mathfrak{r}} \in \QQ).
\]
\item Let $\phi : S^{\ZZ}_{\lambda} \to \QQ\sym{n}f_{\IT^{\lambda}}$ be the $\QQ\sym{n}$-isomorphism obtained by post-composing the $\QQ\sym{n}$-isomorphism $S^{\ZZ}_{\lambda} \to \QQ\sym{n}Y_{\IT^{\lambda}}$ that sends $e_{\IT^{\lambda}}$ to $Y_{\IT^{\lambda}}$, with right multiplication by $\sigma_{\lambda}^{-1}$.  Then $\phi(e_{\IT^{\lambda}}) = f_{\IT^{\lambda}}$, and
\[
	f_{\s} = \phi(e_{\s}) + \sum_{\substack{\mathfrak{r} \in \STab(\lambda) \\ \mathfrak{r} \rhd \s}} b_{\mathfrak{r}} \phi(e_{\mathfrak{r}}) \qquad ( b_{\mathfrak{r}} \in \QQ).
\]
\end{enumerate}
\end{prop}
	
\begin{proof} \hfill
\begin{enumerate}
\item If $\mathfrak{v} \ne \s{\downarrow_m}$, then
		$E_{\mathfrak{v}}E_{\s \downarrow_{m}} = 0$ since these are distinct orthogonal idempotents.  Consequently $E_{\mathfrak{v}}E_{\s}= 0$, and thus
		$f_{\mathfrak{u},\mathfrak{v}}f_{\s,\t}=0$.
		
		If $\mathfrak{v} = \s{\downarrow_m}$, then $[\mu]$ is a subdiagram
		of $[\lambda]$.
		Let $\mathfrak{p}$ be the skew $\lambda/\mu$-tableau obtained by
		removing $\mathfrak{v}$ from $\s$. For a $\mu$-tableau $\mathfrak{w}$,
		let $\mathfrak{w}\sqcup\mathfrak{p}$
		be the $\lambda$-tableau obtained by patching $\mathfrak{p}$ and
		$\mathfrak{w}$ together (thus $\s=
		\mathfrak{v}\sqcup\mathfrak{p}$ and $\s'=\mathfrak{u}\sqcup\mathfrak{p}$).
		Since $\mathfrak{p}$ takes values
		in $\{m+1,\ldots, n\}$, it follows that
		$\mathfrak{w}\sqcup\mathfrak{p}$ is standard whenever
		$\mathfrak{w}$ is standard.
		By Theorem \ref{thm:basics-on-young-basis}(4), if $\sigma\in \sym{m}$,
		and $\mathfrak{w} \in \STab(\mu)$ such that
		if $\sigma f_{\mathfrak{v},\mathfrak{w}}
		=\sum_{\mathfrak{v}_i\in\STab(\mu)} a_i f_{\mathfrak{v}_i,\mathfrak{w}}$ for $a_i\in \QQ$, then the coefficients $a_i$ are independent of
		the choice of $\mathfrak{w}$, and moreover
		$
		\sigma f_{\mathfrak{v}\sqcup\mathfrak{p},\mathfrak{q}} = \sum_{\mathfrak{v}_i\in\STab(\mu)} a_i f_{\mathfrak{v_i}\sqcup\mathfrak{p},\mathfrak{q}}
		$ for any $\mathfrak{q}\in \STab(\lambda)$.  Since this is true for all $\sigma \in \sym{m}$, it remains true when we replace $\sigma$ by any element of $\QQ\sym{m}$.
		Since $f_{\mathfrak{u},\mathfrak{v}}\in \QQ\sym{m}$,
		and $f_{\mathfrak{u},\mathfrak{v}}f_{\mathfrak{v},\mathfrak{w}}=
		\gamma_{\mathfrak{v}}f_{\mathfrak{u},\mathfrak{w}}$ by Theorem \ref{thm:basics-on-young-basis}(3), it follows that
		\[f_{\mathfrak{u},\mathfrak{v}}f_{\s,\t} =f_{\mathfrak{u},\mathfrak{v}}f_{\mathfrak{v}\sqcup\mathfrak{p},\t} = \gamma_{\mathfrak{v}}
		f_{\mathfrak{u}\sqcup\mathfrak{p},\t}=\gamma_{\mathfrak{v}} f_{\s',\t}\] as desired.
		
		\item The conditions on $[\lambda]$ and $[\mu]$ imply that 	
	 $\add{\s}{\mathfrak{w}} \in \STab(\lambda+\mu)$ for all
		$\mathfrak{w} \in \STab(\mu)$, and that
		$\mathrm{res}_{\add{\s}{\mathfrak{u}}}(n+j)
		=\mathrm{res}_{\mathfrak{u}}(j) +\lambda_1$ for all $1\leq j \leq m$, and thus
		\[
		\mathrm{res}_{\add{\s}{\mathfrak{u}}}(n+i+1)-
		\mathrm{res}_{\add{\s}{\mathfrak{u}}}(n+i)
		=\mathrm{res}_{\mathfrak{u}}(i+1)-\mathrm{res}_{\mathfrak{u}}(i)
		\]
		for any $1 \leq i \leq m-1$.  It is enough to prove part (2) for the case when $\sigma=s_i$ for
		some $1\leq i\leq m-1$, and since $s_i^{+n}=s_{i+n}$, this follows directly from Theorem \ref{thm:basics-on-young-basis}(4).

		\item 
The statement is trivial for $\s = \IT^{\lambda}$, and so let $\s \ne \IT^{\lambda}$.  Then, since $\s$ is standard, there exists some $i$ such that $i+1$ lies on a higher row of $\s$ than $i$.  Let $\s' = s_i \cdot \s$.  Then $\s' \in \STab(\lambda)$, and $\s' \rhd \s$ \cite[Lemma 3.7]{Mathasbook99}, so that premultiplying any reduced expression of $d(\s')$ by $s_i$ will yield a reduced expression for $d(\s)$. By induction, we have
		$$
		d(\s) f_{\IT^{\lambda}, \t} = s_i d(\s') f_{\IT^{\lambda}, \t} =
		s_i(f_{\s',\t} + \sum_{\substack{\mathfrak{r} \in \STab(\lambda) \\ \mathfrak{r} \rhd \s'}} a_{\mathfrak{r}} f_{\mathfrak{r}, \t}) \qquad (a_{\mathfrak{r}}\in\QQ)
			$$
        If $s_i \cdot \mathfrak{r} \in \STab(\lambda)$ for some $\mathfrak{r}\in\STab(\lambda)$ with $\mathfrak{r}\rhd \s'$, then, since $d(\mathfrak{r})\rhd d(\s')$ (\cite[Theorem 3.8]{Mathasbook99}, we have that $d(\mathfrak{r})$ has a reduced expression which is a proper subexpression of a reduced expression of $d(\s')$, so that $s_id(\mathfrak{r})$ has a reduced expression which is a proper subexpression of a reduced expression of $s_id(\s')=d(\s)$, or equivalently, $s_i\cdot\mathfrak{r}\rhd \s$.  The proof is now complete by applying Theorem \ref{thm:basics-on-young-basis}(4).

			
		\item We have $\phi(e_{\IT^{\lambda}}) = Y_{\IT^{\lambda}} \sigma_{\lambda}^{-1} = f_{\IT^{\lambda}}$ by Theorem \ref{thm:basics-on-young-basis}(5).  Furthermore, by part (3),
		$$
		\phi(e_{\s}) = \phi(d(\s) e_{\IT^{\lambda}}) = d(\s) f_{\IT^{\lambda}} = f_{\s} + \sum_{\mathfrak{r} \rhd \s} a_{\mathfrak{r}} f_{\mathfrak{r}}.$$
		Thus, $f_{\s} = \phi(e_{\s}) - \sum_{\mathfrak{r} \rhd \s} a_{\mathfrak{r}} f_{\mathfrak{r}}$, and the desired result follows by induction.
		\end{enumerate}
\end{proof}

\subsection{Gcds}
 \begin{defn}
 Let $\mathcal{L}$ be a free $\ZZ$-module of finite
 rank. We define, for each non-zero $z\in \mathcal{L}\otimes_{\ZZ}\QQ$,
 a positive rational number $\mathrm{gcd}_{\mathcal{L}}(z)$ as follows:
 \[
\mathrm{gcd}_{\mathcal{L}}(z) := \sup\{\kappa\in \QQ\mid z/\kappa \in \mathcal{L}\} \in \QQ \cup \{\infty\}.
 \]
 \end{defn}

We gather together some elementary properties resulting from this definition.

\begin{lem} \label{lem:basic-gcds}
	Let $\mathcal{L}$ be a free $\ZZ$-module of finite
	rank, and let $z\in \mathcal{L} \otimes_{\ZZ} \QQ$.
\begin{enumerate}
\item If $\{v_1,\dotsc, v_n \}$ is a $\ZZ$-basis for $\mathcal{L}$, so that it is a $\mathbb{Q}$-basis for $\mathcal{L} \otimes_{\ZZ} \QQ$, and $z = \sum_{i=1}^n \frac{a_i}{b_i} v_i \ne 0$ where $\gcd(a_i,b_i)=1$ for each $1\leq i\leq n$, then
 $\gcd_{\mathcal{L}}(z)$ is the (positive) greatest common divisor of the integers $a_1, \dotsc, a_n$ divided by the (positive) least common multiple of the integers $b_1,\dotsc, b_n$.

 In particular, $\gcd_{\mathcal{L}}(z) = \max\{\kappa\in \QQ\mid z/\kappa \in \mathcal{L}\}$.

\item If $a \in \QQ_{>0}$, then $\gcd_{\mathcal{L}}(az) = a\gcd_{\mathcal{L}}(z)$.

\item We have $z \in \mathcal{L} \otimes_{\ZZ} \ZP$ if and only if $\gcd_{\mathcal{L}}(z) \in \ZP$.

\item Let $\mathcal{K}$ be a direct $\ZZ$-summand of $\mathcal{L}$, and suppose that $z\in \mathcal{K}\otimes_{\ZZ} \QQ$.
 Then $\mathrm{gcd}_{\mathcal{K}}(z)=\mathrm{gcd}_{\mathcal{L}}(z)$.
 \end{enumerate}
 \end{lem}

 \begin{proof}
 Part (1) follows directly from the definition of $\gcd_{\mathcal{L}}$, while parts (2) and (3) follow immediately from part (1).

 For part (4), since $\mathcal{K}$ is a direct summand of $\mathcal{L}$, it follows that $\mathcal{L}=\mathcal{K}\oplus \mathcal{K}'$ for some submodule $\mathcal{K}'$ of $\mathcal{L}$. With respect to this decomposition, $z$ regarded as an element of $\mathcal{L}\otimes_{\ZZ}\QQ$ is realized as $(z,0)$. Thus
 for any $\kappa\in \QQ$, $(z,0)/\kappa = (z/\kappa , 0)$, and in particular $\mathrm{gcd}_{\mathcal{K}}(z) = \mathrm{gcd}_{\mathcal{L}}(z)$.
 \end{proof}

\begin{rem} \hfill
\begin{enumerate}
	\item One can give an alternative definition of $\gcd_{\mathcal{L}}$ using Lemma \ref{lem:basic-gcds}(1), but our definition makes it clear that $\gcd_{\mathcal{L}}$ is independent of the basis chosen for $\mathcal{L}$.
	\item The condition that $\mathcal{K}$ is a direct summand, instead of merely a submodule, in Lemma
 \ref{lem:basic-gcds}(4) is necessary. For example, $2\ZZ$ is a submodule, but not a direct summand of $\ZZ$, and for
 $1\in 2\ZZ\otimes_{\ZZ}\QQ\cong \QQ$, we have $\mathrm{gcd}_{2\ZZ}(1) = 1/2$,
 while $\mathrm{gcd}_{\ZZ}(1) = 1$.
 \item It is clear from Lemma \ref{lem:basic-gcds}(1) that $\gcd_{\mathcal{L}}(z)$ generalizes the greatest common divisor of the coefficients of $z$ (which makes sense when the latter are integers).  
\end{enumerate}
\end{rem}

Recall Young's seminormal basis $\{ f_{\s} \mid \s \in \STab(\lambda) \}$ for the dual Specht module $S^{\QQ}_{\lambda}$ defined in Theorem \ref{thm:basics-on-young-basis}(6).

\begin{lem} \label{lem:denominator-young-basis}
Let $\lambda \vdash n$ and $\s \in \STab(\lambda)$.  Then
 $\gcd_{\ZZ\sym{n}}(f_{\s}) = \frac{1}{\dd_{\s}}$ for some $\dd_{\s} \in \ZZ^+$. We call $\dd_\s$ the {\em denominator} of $f_\s$.
\end{lem}

\begin{proof}
By Theorem \ref{thm:basics-on-young-basis}(5), $f_{\IT^\lambda}=Y_{\IT^\lambda}\sigma_\lambda^{-1}\in\ZZ\sym{n}$. Thus, since $\ZZ\sym{n}Y_{\IT^{\lambda}}$ is a $\ZZ$-summand of $\ZZ\sym{n}$, so is $\ZZ\sym{n}f_{\IT^{\lambda}}$.  Consequently $\gcd_{\ZZ\sym{n}}(f_{\s}) = \gcd_{\ZZ\sym{n}f_{\IT^{\lambda}}} (f_{\s})$ by Lemma \ref{lem:basic-gcds}(4). Using the $\QQ\sym{n}$-isomorphism $\phi$ in Proposition \ref{prop:products-of-young-basis-of-different-size}(4), we have $\phi(S^{\ZZ}_{\lambda}) = \ZZ\sym{n}f_{\IT^{\lambda}}$ and $\phi^{-1}(f_{\s}) = e_{\s} + \sum_{\mathfrak{r} \rhd \s} b_{\mathfrak{r}} e_{\mathfrak{r}}$ where $b_{\mathfrak{r}}\in\QQ$.  Thus,
$$
\gcd\nolimits_{\ZZ\sym{n}f_{\IT^{\lambda}}} (f_{\s}) = \gcd\nolimits_{S^{\ZZ}_{\lambda}} (\phi^{-1}(f_{\s})) = \frac{1}{\dd_{\s}},$$
for some $\dd_{\s} \in \ZZ^+$, by Lemma \ref{lem:basic-gcds}(1).
\end{proof}


\section{Main results} \label{sec:main results}

\subsection{Comparison of Jantzen filtrations}
Keep the notations in Subsection \ref{subsec:Weyl-module}. Let $G$ be a connected reductive algebraic group over $\FF$, obtained from a split connected reductive group $G_{\ZZ}$. Let $\lambda$ and $\mu$ be dominant integral weights.  The canonical $G$-morphism $$\iota_{\lambda,\mu} : \Delta(\lambda+\mu) \to \Delta(\lambda) \otimes \Delta(\mu)$$ is characterized by $\iota_{\lambda,\mu}(\eta_{\lambda+\mu}) = \eta_{\lambda} \otimes \eta_{\mu}$, where $\eta_{\nu}$ is the highest weight vector generating the Weyl module $\Delta(\nu)$ for any dominant integral weight $\nu$.  This $G$-morphism certainly depends on the choice of the highest weight vectors in the Weyl modules, but is unique up to scalars. It is injective, and sometimes admits a splitting $\psi_{\lambda,\mu}$ (i.e.\ a $G$-morphism satisfying $\psi_{\lambda,\mu} \circ \iota_{\lambda,\mu}= \mathrm{id}$).
We say that $\psi_{\lambda,\mu}$ is {\em defined over $\ZP$} if $\psi_{\lambda,\mu}$ is induced from a $G_{\ZP}$-morphism  $$\psi^{\ZP}_{\lambda,\mu} : \Delta_{\ZP} (\lambda) \otimes \Delta_{\ZP}(\mu) \to \Delta_{\ZP}(\lambda+\mu).$$
By abusing notation, we shall also write $\psi_{\lambda,\mu}$ for $\psi^{\ZP}_{\lambda,\mu}$ in what follows.


As our first main result which also serves as a motivation to the study of the split condition for $\iota_{\lambda,\mu}$, we have the following result that compares the Jantzen filtrations of $\Delta_{\ZP}(\lambda)$ and $\Delta_{\ZP}(\lambda+\mu)$.

\begin{thm} \label{thm:split-Jantzen-filtration}
Let $G$ be a connected reductive algebraic group over $\FF$, obtained from a split connected reductive group $G_{\ZZ}$. Let $\lambda$ and $\mu$ be dominant integral weights. If the canonical $G$-morphism $\iota_{\lambda,\mu}:\Delta(\lambda+\mu)\to \Delta(\lambda)\otimes \Delta(\mu)$ admits a splitting $\psi_{\lambda,\mu}$ defined over $\ZZ_{(p)}$, then there are injective $\ZZ_{(p)}$-linear maps
\[
\psi_{\lambda,\mu}\circ\iota_\lambda:\Delta_{\ZZ_{(p)}}(\lambda)^i \to \Delta_{\ZZ_{(p)}}(\lambda+\mu)^i
\]
for all $i\geq 0$, where $\iota_\lambda$ is the map $\Delta_{\ZZ_{(p)}}(\lambda) \to \Delta_{\ZZ_{(p)}}(\lambda)\otimes\Delta_{\ZZ_{(p)}}(\mu)$ given by $\iota_\lambda(x) = x\otimes \eta_\mu$, and $\eta_\mu$ is the highest weight vector in $\Delta_{\ZZ_{(p)}}(\mu)$.
\end{thm}

\begin{proof}
By assumption, $\psi_{\lambda,\mu}\circ\iota_{\lambda,\mu}=\mathrm{id}$, and so we have the decomposition $\Delta_{\ZZ_{(p)}}(\lambda)\otimes \Delta_{\ZZ_{(p)}}(\mu) \cong \Delta_{\ZZ_{(p)}}(\lambda+\mu)\oplus \mathrm{ker} (\psi_{\lambda,\mu})$ as $\mathrm{Dist}(G_{\ZP})$-modules. Recall the contravariant symmetric bilinear forms $c_\lambda, c_\mu$ and $c_{\lambda+\mu}$ introduced in Subsection \ref{subsec:Weyl-module}. The tensor product $c_\lambda\otimes c_\mu$ defines a symmetric bilinear form on $\Delta_{\ZP}(\lambda)\otimes \Delta_{\ZZ_{(p)}}(\mu)$. Since the Cartan involution $\tau$ commutes with the comultiplication $\Delta$ on $\mathrm{Dist}(G_{\ZP})$, i.e., $(\tau\otimes \tau)\Delta = \Delta \tau$, it follows that $c_\lambda\otimes c_\mu$ and its restriction to $\Delta_{\ZZ_{(p)}}(\lambda+\mu)$ are symmetric and contravariant. As a result, $c_\lambda\otimes c_\mu$ coincides with $c_{\lambda+\mu}$ on $\Delta_{\ZZ_{(p)}}(\lambda+\mu)$ as
$$(c_\lambda\otimes c_\mu)(\iota_{\lambda,\mu}(\eta_{\lambda+\mu}), \iota_{\lambda,\mu}(\eta_{\lambda+\mu}))=(c_\lambda\otimes c_\mu)(\eta_{\lambda}\otimes \eta_{\mu},\eta_{\lambda}\otimes \eta_\mu)=c_\lambda(\eta_\lambda,\eta_\lambda)c_\mu(\eta_\mu,\eta_\mu)=1.$$

We claim that $(c_\lambda\otimes c_\mu)(\iota_{\lambda,\mu}(u\eta_{\lambda+\mu}),w)=0$ for any $u\in \mathrm{Dist}(G_{\ZP})$ and $w\in \mathrm{ker}(\psi_{\lambda,\mu})$.
Indeed, $\lambda+\mu$ is the highest weight in both $\Delta_{\ZZ_{(p)}}(\lambda+\mu)$ and $\Delta_{\ZZ_{(p)}}(\lambda)\otimes \Delta_{\ZZ_{(p)}}(\mu)$, and the corresponding weight spaces are of rank one. It follows that all weights in $\mathrm{ker}(\psi_{\lambda,\mu})$ are strictly smaller than $\lambda+\mu$, and thus $(c_\lambda\otimes c_\mu)(\eta_\lambda\otimes \eta_\mu, \mathrm{ker}(\psi_{\lambda,\mu}))=0$. So, for any $u\in \mathrm{Dist}(G_{\ZP})$ and $w\in \mathrm{ker}(\psi_{\lambda,\mu})$, we have \[(c_\lambda\otimes c_\mu)(\iota_{\lambda,\mu}(u\eta_{\lambda+\mu}),w)=(c_\lambda\otimes c_\mu)(u\cdot(\eta_{\lambda}\otimes\eta_\mu),w)=(c_\lambda\otimes c_\mu)(\eta_{\lambda}\otimes\eta_\mu,\tau(u) w)=0.\] 

 Now for $u\eta_\lambda\in \Delta_{\ZZ_{(p)}}(\lambda)$, we have
 $\iota_\lambda(u\eta_\lambda)=u\eta_\lambda\otimes \eta_\mu = \iota_{\lambda,\mu} \psi_{\lambda,\mu}(u\eta_\lambda \otimes \eta_\mu) + w$ for some $w\in \mathrm{ker}(\psi_{\lambda,\mu})$, and for any $u'\eta_{\lambda+\mu} \in \Delta_{\ZZ_{(p)}}(\lambda+\mu)$, using the claim above,
\begin{align*}
    c_{\lambda+\mu}(\psi_{\lambda,\mu}\iota_\lambda(u\eta_\lambda),
    u'\eta_{\lambda+\mu})
    &=
    (c_\lambda\otimes c_\mu)(\iota_{\lambda,\mu}\psi_{\lambda,\mu}
    \iota_{\lambda}(u\eta_{\lambda}),\iota_{\lambda,\mu}(u'\eta_{\lambda+ \mu}))\\
    &=
    (c_\lambda\otimes c_\mu)(\iota_{\lambda,\mu}\psi_{\lambda,\mu}
    (u\eta_\lambda\otimes \eta_\mu),\iota_{\lambda,\mu}(u'\eta_{\lambda+\mu}))\\
    &=
    (c_\lambda\otimes c_\mu)(\iota_{\lambda}(u\eta_\lambda),\iota_{\lambda,\mu}(u'\eta_{\lambda+\mu}))\\
    &= (c_\lambda\otimes c_\mu)(u\eta_\lambda\otimes
    \eta_\mu, u'(\eta_\lambda\otimes \eta_\mu))\\
    & = c_\lambda(u\eta_\lambda, u'\eta_\lambda).
\end{align*}
Consequently, if $u\eta_\lambda\in \Delta_{\ZZ_{(p)}}(\lambda)^i$, then $\psi_{\lambda,\mu}\iota_\lambda(u\eta_\lambda)\in \Delta_{\ZZ_{(p)}}(\lambda+\mu)^i$.
It remains to show that
$\psi_{\lambda,\mu}\iota_\lambda:
 \Delta_{\ZZ_{(p)}}(\lambda)\to  \Delta_{\ZZ_{(p)}}
 (\lambda+\mu)$ is injective, i.e.,
 $\psi_{\lambda,\mu}\iota_{\lambda}(
 u\eta_\lambda)
 \neq 0$ for $u\eta_\lambda\neq 0$. We may take
 $u'\eta_\lambda \in  \Delta_{\ZZ_{(p)}}(\lambda)$ for some $u'\in \mathrm{Dist}(G_{\ZP})$
 such that $c_\lambda(u\eta_\lambda, u'\eta_\lambda)\neq 0$.
 Here we use the fact that $c_\lambda$ is non-degenerate
 over the field of fractions of $\ZP$. Then the identity above reads
 $c_{\lambda+\mu}(\psi_{\lambda,\mu}\iota_{\lambda}(
 u\eta_\lambda), u'\eta_{\lambda+\mu})
 =c_\lambda(u\eta_\lambda, u'\eta_\lambda)\neq 0$.
 So $\psi_{\lambda,\mu}\iota_\lambda(u\eta_\lambda)\neq 0$
 as desired.

\end{proof}

\begin{rem} \label{rem:Jantzen} 
The injective $\ZP$-linear map $\psi_{\lambda,\mu}
 \iota_{\lambda}:  \Delta_{\ZZ_{(p)}}(\lambda)^i\to
 \Delta_{\ZZ_{(p)}}(\lambda+\mu)^i$ clearly induces
 an $\FF$-linear map $\Delta(\lambda)^i\to\Delta(\lambda+\mu)^i$.  But the latter  
is not necessarily injective.
For example, when $p=3$, the canonical $\mathrm{SL}_2(\FF)$-morphism
 $\Delta(5)\to \Delta(3)\otimes \Delta(2)$ admits a
 splitting $\psi_{3,2}$ defined over $\ZZ_{(3)}$ by Theorem \ref{thm:split-condition-main} (here the weight $n\omega$, where $\omega$ is the unique fundamental weight of $\mathrm{SL}_2(\FF)$, is denoted simply by $n$). In
 $\Delta_{\ZZ_{(3)}}(3)\otimes \Delta_{\ZZ_{(3)}}(2)$,
 we have
 $$f_\alpha \eta_3\otimes \eta_2 = \tfrac{3}{5}f_\alpha(\eta_3\otimes \eta_2)+ (-\tfrac{3}{5}\eta_2 \otimes f_\alpha\eta_2
 + \tfrac{2}{5}f_\alpha\eta_3 \otimes \eta_2).$$
 Therefore, $\psi_{3,2}(f_\alpha \eta_3\otimes \eta_2) =\tfrac{3}{5}f_\alpha \eta_5$.
 In particular, the image of $f_\alpha \eta_3$ under $\psi_{3,2}\circ \iota_3$, in $\Delta(5)$, is zero.


\end{rem}

We thank H.\ Andersen for communicating to us the following result, which gives a necessary split  condition for $\iota_{\lambda,\mu}$ over $\FF$, which is certainly a necessary split condition for $\iota_{\lambda,\mu}$ over $\ZZ_{(p)}$. 

 \begin{prop}[Andersen] \label{prop:andersen-criterion}
 Let $G$ be a connected, simply-connected and semisimple algebraic group of rank $n$ over $\FF$, and let $\omega_1,\dotsc, \omega_n$ be its fundamental weights. Let
 $\lambda= \sum_{i=1}^n \lambda_i \omega_i$ and $\mu= \sum_{i=1}^n \mu_i \omega_i$ be two dominant integral weights for
 $G$. Then the canonical $G$-morphism
 $\iota_{\lambda,\mu} : \Delta(\lambda+\mu)\to \Delta(\lambda)\otimes \Delta(\mu)$
 splits over $\FF$ only if $p\nmid \binom{\lambda_i+\mu_i}{\lambda_i}$
 for all $i$.
 \end{prop}

 \begin{proof}
 Let $U_{\FF}$ be the hyperalgebra of $G$, and for each
 $1\leq i\leq n$, let $U_{\FF}^i$ be the $\FF$-subalgebra of
 $U_{\FF}$ generated by the divided powers $
 e_{\alpha_i}^{(m)}$, $f_{\alpha_i}^{(m)}$ and
 $\binom{h_i}{m}$ for all $m \in \ZZ^+$. Let $\Delta^{\alpha_i}(\lambda)$
 be the $U_{\FF}^i$-submodule of $\Delta(\lambda)$ generated
 by the highest weight vector $\eta_\lambda$ and let
 $\Delta^{\alpha_i}(\mu)$ and $\Delta^{\alpha_i}(\lambda+\mu)$
 be defined similarly.
 For each $1\leq i\leq n$, there is a commutative
 diagram of $U_{\FF}^i$-modules as follows, where $\varphi_i$ is the canonical morphism:
 \[
 \begin{CD}
 \Delta^{\alpha_i}(\lambda+\mu) @>{\varphi_i}>> \Delta^{\alpha_i}(\lambda)\otimes \Delta^{\alpha_i}(\mu) \\
 @VVV @VVV  \\
 \Delta(\lambda+\mu) @>\iota_{\lambda,\mu}>> \Delta(\lambda)\otimes \Delta(\mu)
 \end{CD}
 \]
 Note that $U_{\FF}^i$ is canonically isomorphic to the
 hyperalgebra of $\mathrm{SL}_2(\FF)$. Moreover, under this identification,
 $\Delta^{\alpha_i}(\lambda)$ is the same as the Weyl
 module $\Delta(\lambda_i)$ for $\mathrm{SL}_2(\FF)$, and
 $\varphi_i$ is the same as the canonical $\mathrm{SL}_2(\FF)$-morphism
 $\Delta(\lambda_i+\mu_i)\to \Delta(\lambda_i)\otimes
 \Delta(\mu_i)$. Also note that if $\iota_{\lambda,\mu}$ admits a
 splitting $\psi_{\lambda,\mu}$, then the image of
 $\Delta^{\alpha_i}(\lambda)\otimes \Delta^{\alpha_i}(\mu)$
 under $\psi_{\lambda,\mu}$ lies in the $U_{\FF}^i$-submodule
 $\Delta^{\alpha_i}(\lambda+\mu)$ of $\Delta(\lambda+\mu)$,
 i.e., $\psi_{\lambda,\mu}$ restricts down to give a
 splitting of $\varphi_i$. Now by \cite[\S4.8(12)]{Donkin98} (or Proposition \ref{prop:Donkin-criterion}, see later), $\varphi_i$ admits a splitting if and only if $p\nmid \binom{\lambda_i+\mu_i}{\lambda_i}$, and so $\iota_{\lambda,\mu}$ splits only if $p\nmid \binom{\lambda_i+\mu_i}{\lambda_i}$ for all $i$.
 \end{proof}

In view of Proposition \ref{prop:andersen-criterion}, we shall refer to the condition $p\nmid {\lambda_i+\mu_i\choose \lambda_i}$ for all $i$ as \textit{Andersen's condition} (on $p$, for the fixed pair $(\lambda,\mu)$ of dominant integral weights.)

\subsection{Split condition for $\iota_{\lambda,\mu}$ for type $A$}

The rest of the section is devoted to the split condition for $\iota_{\lambda,\mu}$ when $G$ is the general linear group $\mathrm{GL}_N(\FF)$, for which we are able to relate the split condition to the product of certain Young symmetrizers and to the denominator of a certain Young's seminormal basis vector.  The weights $\lambda$ and $\mu$ of $G$ in this subsection are polynomial weights, written as partitions with at most $N$ parts.
We recall first the following result, essentially due to Donkin.

 \begin{prop} \label{prop:Donkin-criterion}
 Let $m,n \in \ZZ^+$, and let $\lambda = (m)$ and $\mu = (n)$.  The following statements are equivalent:
\begin{enumerate}
\item[(i)] $\iota_{\lambda,\mu}$ splits over $\ZP$;
\item[(ii)] $\iota_{\lambda,\mu}$ splits over $\FF$;
\item[(iii)] $p \nmid \binom{m+n}{n}$.
\end{enumerate}
 \end{prop}

\begin{proof} The equivalence of (ii) and (iii) is proved by Donkin; see \cite[\S4.8(12)]{Donkin98} for its quantized analogue.  As the splitting map constructed by Donkin is defined over $\ZP$, this proves that (ii) implies (i).  That (i) implies (ii) is obvious.
\end{proof}

\begin{rem}
The splitting of $\iota_{\lambda,\mu}$ over $\ZP$ and over $\FF$ seems to be intimately related, as suggested by Proposition \ref{prop:Donkin-criterion}.  The former certainly implies the latter, and we know of no example which shows that the latter does not imply the former.
\end{rem}

To state the condition for which $\iota_{\lambda,\mu}$ splits over $\ZZ_{(p)}$, we make the following definition:

\begin{defn}
Let $\lambda \vdash n$ and $\mu \vdash m$.  Define $\theta_{\lambda,\mu} \in \mathbb{Z}^+$ by
$$
\theta_{\lambda,\mu} := \gcd\nolimits_{\ZZ\sym{n+m}} (Y_{\IT^{\lambda}}Y_{(\IT^{\mu})^{+n}} Y_{\add{\IT^{\lambda}}{\IT^{\mu}}}).
$$
\end{defn}
	
We first collect together some properties of $\theta_{\lambda,\mu}$.

\begin{lem} \label{lem:lem-for-product-of-Young-symmetrizers}
Let $\lambda \vdash n$, $\mu \vdash m$, and let $\s \in \Tab(\lambda)$ and $\t \in \Tab(\mu)$.  Then
\begin{enumerate}
\item
$\theta_{\lambda,\mu} = \gcd\nolimits_{\ZZ\sym{n+m}} (Y_{\s}Y_{\t^{+n}} Y_{\add{\s}{\t}}) = \gcd\nolimits_{\ZZ\sym{n+m}} ([C_{\add{\s}{\t}}] Y_{\add{\s}{\t}} \{ R_{\s} R_{\t^{+n}} \})$,
\item $\theta_{\lambda,\mu} = \theta_{\mu,\lambda}$, and
\item $\theta_{\lambda,\mu} \mid \phl^{\lambda+\mu}$.
\end{enumerate}
\end{lem}

 \begin{proof}
	Under the anti-automorphism of $\ZZ\sym{n+m}$ induced by
	the inverse operator on $\sym{n+m}$, the image of
	$Y_{\s}Y_{\t^{+n}}Y_{\add{\s}{\t}} = \{ R _{\s} \} [C_{\s}] \{ R_{\t^{+n}} \} [C_{\t^{+n}} ]\{R_{\add{\s}{\t}}\}[C_{\add{\s}{\t}}]$
	is
	\begin{align*}
	[C_{\add{\s}{\t}}]\{R_{\add{\s}{\t}}\}[C_{\t^{+n}}]\{R_{\t^{+n}}\}
	[C_{\s}]\{R_{\s}\} &= [C_{\add{\s}{\t}}]\{R_{\add{\s}{\t}}\}[C_{\t^{+n}}][C_{\s}]\{R_{\t^{+n}}\}\{R_{\s}\} \\
	&= [C_{\add{\s}{\t}}] Y_{\add{\s}{\t}} \{ R_{\s} R_{\t^{+n}} \},
	\end{align*}
	since $C_{\s}$ commutes with $\{R_{\t^{+n}}\}$, and $C_{\add{\s}{\t}} = C_{\s}C_{\t^{+n}}$.
	In particular, they have the same $\gcd$ in $\ZZ\sym{n+m}$.  This proves the second equality in part (1).
	
	For the first equality in part (1), let $h = d(\s) d(\t)^{+n}\in  \sym{n+m}$, so that $\s=
	h \cdot \IT^{\lambda}$ and $\t^{+n}= h \cdot (\IT^{\mu})^{+n}$, and hence $\add{\s}{\t} = h \cdot (\add{\IT^{\lambda}}{\IT^{\mu}})$.  Thus $Y_{\s}Y_{\t^{+n}} Y_{\add{\s}{\t}} = h (Y_{\IT^{\lambda}}Y_{(\IT^{\mu})^{+n}} Y_{\add{\IT^{\lambda}}{\IT^{\mu}}}) h^{-1}$, and so $Y_{\s}Y_{\t^{+n}} Y_{\add{\s}{\t}}$ has the same gcd as $Y_{\IT^{\lambda}}Y_{(\IT^{\mu})^{+n}} Y_{\add{\IT^{\lambda}}{\IT^{\mu}}}$ over $\ZZ\sym{n+m}$.
	
	For part (2), let $\tau \in \sym{n+m}$ such that $\tau(i) = i+m$ if $i \leq n$, and $\tau(i) = i-n$ otherwise.  Then $\tau \cdot \s = \s^{+m}$ and $\tau \cdot \t^{+n} = \t$.  Furthermore, $\tau \cdot (\add{\s}{\t})$ is a column-wise rearrangement of $\add{\t}{\s}$; consequently, $Y_{\tau\cdot(\add{\s}{\t})}=Y_{\add{\t}{\s}}$. Thus
	$\tau (Y_{\s}Y_{\t^{+n}} Y_{\add{\s}{\t}}) \tau^{-1}=Y_{\t}Y_{\s^{+m}} Y_{\tau\cdot(\add{\s}{\t})}=Y_{\t}Y_{\s^{+m}} Y_{\add{\t}{\s}}$, and so $Y_{\t}Y_{\s^{+m}} Y_{\add{\t}{\s}}$ has the same gcd as $Y_{\s}Y_{\t^{+n}} Y_{\add{\s}{\t}}$, giving $\theta_{\lambda,\mu} = \theta_{\mu,\lambda}$.
	
	For part (3),
	note that $C_{\add{\s}{\t}} = C_{\s} C_{\t^{+n}}$, while $R_{\add{\s}{\t}}$ contains $R_{\s}R_{\t^{+n}}$.  Let $\Gamma$ be a left transversal of $R_{\s}R_{\t^{+n}}$ in $R_{\add{\s}{\t}}$.  Then
	$$
	Y_{\add{\s}{\t}} = \{ R_{\add{\s}{\t}} \} [ C_{\add{\s}{\t}}] = \{ \Gamma R_{\s}R_{\t^{+n}} \} [ C_{\s} C_{\t^{+n}} ] = \{ \Gamma \} \{ R_{\s} \} \{ R_{\t^{+n}} \} [C_{\s}][C_{\t^{+n}}] = \{ \Gamma \} Y_{\s}Y_{\t^{+n}}.$$
	Thus,
	$$
	\phl^{\lambda+\mu} Y_{\add{\s}{\t}} = (Y_{\add{\s}{\t}})^2 = \{\Gamma\} Y_{\s}Y_{\t^{+n}} Y_{\add{\s}{\t}},$$
	and so, since $\{ \Gamma \} \in \ZZ\sym{n+m}$ and $\gcd_{\ZZ\sym{n+m}} (Y_{\add{\s}{\t}}) = 1$, we have
	$$\gcd\nolimits_{\ZZ\sym{n+m}} (Y_{\s}Y_{\t^{+n}}Y_{\add{\s}{\t}}) \mid \gcd\nolimits_{\ZZ\sym{n+m}} (\{\Gamma\} Y_{\s}Y_{\t^{+n}}Y_{\add{\s}{\t}}) =  \gcd\nolimits_{\ZZ\sym{n+m}} (\phl^{\lambda+\mu} Y_{\add{\s}{\t}}) = \phl^{\lambda+\mu}.$$
\end{proof}

We now state our result on the split condition for $\iota_{\lambda,\mu}$ for type $A$.

\begin{thm}\label{thm:split-condition-main}
Let $\lambda\vdash n$, $\mu\vdash m$ and $N \in \ZZ^+$ with $N\geq n+m$. The canonical $\mathrm{GL}_N(\FF)$-morphism $\iota_{\lambda,\mu}:\Delta(\lambda+\mu)\to \Delta(\lambda)\otimes \Delta(\mu)$ admits a splitting defined over $\ZZ_{(p)}$ if and only if $p\nmid \frac{\phl^{\lambda+\mu}}{\theta_{\lambda,\mu}}$.
\end{thm}

To prove this theorem, we need some preparations.  Let $G_{\ZZ}=\mathrm{GL}_N(\ZZ)$ and $E_{\ZZ}$ be a free $\ZZ$-module of rank $N$ with a $\ZZ$-basis $\{v_1,\ldots, v_N\}$.
The $r$-th tensor power $E_{\ZZ}^{\otimes r}$ is a $(\ZZ G_{\ZZ}, \ZZ \sym{r})$-bimodule, where $\sym{r}$ acts on the right by place permutation, with a $\ZZ$-basis
$$\mathcal{B} := \{ v_{f} := v_{f(1)} \otimes \dotsb \otimes v_{f(r)} \mid f : \{ 1,\dotsc, r\} \to \{ 1,\dotsc, N \}\, \}.$$  The right action of $\sym{r}$ on $E_{\ZZ}^{\otimes r}$ restricts to give an action on $\mathcal{B}$ via $v_{f} \cdot \sigma = v_{f \circ \sigma}$, so that $\mathcal{B}$ is a disjoint union of $\sym{r}$-orbits.  This induces a decomposition of $E_{\ZZ}^{\otimes r}$ into a direct sum of right $\ZZ\sym{r}$-submodules, where each summand is indexed by an $\sym{r}$-orbit.  In particular, when $N\geq r$, so that each $\sigma \in \sym{r}$ may be viewed as a function from $\{1,\dotsc, r\}$ to $\{ 1,\dotsc, N \}$, $\mathcal{B}_{\sym{r}} := \{ v_{\sigma} \mid \sigma \in \sym{r} \}$ is such an $\sym{r}$-orbit, and its $\ZZ$-span, denoted $\ZZ \mathcal{B}_{\sym{r}}$, is a $\ZZ\sym{r}$-summand of $E_{\ZZ}^{\otimes r}$ isomorphic to $\ZZ \sym{r}$, where $v_{\sigma}$ is identified with $\sigma$.

Let $H$ be a subgroup of $\sym{r}$, and define
$$
(E_{\ZZ}^{\otimes r} )_H := E_{\ZZ}^{\otimes r} \otimes_{\ZZ \sym{r}} \ZZ \sym{r} \{ H \}.
$$
Since $E_{\ZZ}^{\otimes r}$ decomposes, as a right $\sym{r}$-module, into a direct sum where each summand is indexed by a $\sym{r}$-orbit of $\mathcal{B}$, $(E_{\ZZ}^{\otimes r})_H$ decomposes as a $\ZZ$-module into an analogous direct sum.
We have a $G_{\ZZ}$-morphism $\Phi_{H} : (E_{\ZZ}^{\otimes r} )_H \to E_\ZZ^{\otimes r}$, defined by $ a \otimes \{ H \} \mapsto a \{ H \}$ for $a \in E_{\ZZ}^{\otimes r}$, which respects the abovementioned decomposition.  When $N \geq r$ so that $\ZZ\mathcal{B}_{\sym{r}}$ is a right $\sym{r}$-summand of $E_{\ZZ}^{\otimes r}$, the map $\Phi_{H}$ is injective when restricted to the corresponding summand $(\ZZ\mathcal{B}_{\sym{r}})_H$,  and when post-composed with the isomorphism $\ZZ\mathcal{B}_{\sym{r}} \cong \ZZ\sym{r}$, sends  $(\ZZ\mathcal{B}_{\sym{r}})_H$ bijectively onto $\ZZ\sym{r}\{H\}$, which is a $\ZZ$-summand of $\ZZ\sym{r}$, since $\ZZ\sym{r} = \ZZ\sym{r}\{H\} \oplus \bigoplus_{\sigma \in \sym{r} \setminus \Gamma} \ZZ\sigma$, where $\Gamma$ is a left transversal of $H$ in $\sym{r}$.

 Let $\lambda$ be a partition of $r$ of at most $N$
 parts and let $\s$ be a $\lambda$-tableau. Consider the $G_{\ZZ}$-morphism:
 \[
 \delta^{\ZZ}_{\s}:E_{\ZZ}^{\otimes r}[C_{\s}]\hookrightarrow E_{\ZZ}^{\otimes r}\twoheadrightarrow (E^{\otimes r}_{\ZZ})_{ R_{\s}},
 \]
 where the second map sends $a \in E_{\ZZ}^{\otimes r}$ to $a \otimes \{ R_{\s} \} \in (E^{\otimes r}_{\ZZ})_{ R_{\s}}$.
 Its image $\mathrm{Im}(\delta^{\ZZ}_{\s})$ is isomorphic to the integral dual Weyl module $\nabla_{\ZZ}(\lambda)$ \cite[p.219--220]{ABW}, and
 is a $\ZZ$-summand of $(E_{\ZZ}^{\otimes r})_{R_{\s}}$.
 To describe the highest weight vector in $\mathrm{Im}(\delta^{\ZZ}_{\s})$, let $\underline{1}^{\s} : \{1 ,\dotsc, r \} \to \{1,\dotsc, N\}$ be the function such that $\underline{1}^{\s}(i) = j$ if the node labelled $i$ in $\s$ lies in its $j$-th row.  Then $\delta^{\ZZ}_{\s}(v_{\underline{1}^{\s}}[C_{\s}])$ is a highest weight vector in $\mathrm{Im}(\delta^{\ZZ}_{\s})$.

All the statements in the previous three paragraphs behave well under base change, so that we have entirely analogous statements as above, with $\ZZ$ being replaced by any commutative ring $R$ with 1.


\begin{proof}[Proof of Theorem \ref{thm:split-condition-main}]
Note that $\iota_{\lambda,\mu}$ admits a splitting defined over $\ZZ_{(p)}$ if and only if the canonical $\mathrm{GL}_N({\FF})$-morphism $\phi:\nabla(\lambda)\otimes \nabla(\mu)\to \nabla(\lambda+\mu)$ admits a splitting defined over $\ZZ_{(p)}$.

Let $\s = \IT^{\lambda}$ and $\t= \IT^{\mu}$.  Recall that $[C_{\add{\s}{\t}}] = [C_{\s}] [C_{\t^{+n}}]$, so that $E_R^{\otimes (n+m)} [C_{\add{\s}{\t}}] = E_R^{\otimes n}[C_{\s}]  \otimes E_R^{\otimes m}[C_{\t}]$, for any commutative ring $R$ with $1$.  Thus we have the following $G_{\ZZ}$-morphisms, which are well-behaved under base change:
\[
\zeta_\ZZ: E_{\ZZ}^{\otimes (n+m)}[C_{\add{\s}{\t}}] \xrightarrow{\delta^{\ZZ}_{\add{\s}{\t}}} (E_{\ZZ}^{\otimes (n+m)})_{R_{\add{\s}{\t}}} \xrightarrow{\beta} E_{\ZZ}^{\otimes (n+m)}[C_{\add{\s}{\t}}]
\]
\[
\psi_{\ZZ}: E_{\ZZ}^{\otimes (n+m)}[C_{\add{\s}{\t}}]\xrightarrow{\zeta_\ZZ} E_{\ZZ}^{\otimes (n+m)}[C_{\add{\s}{\t}}] =
E_{\ZZ}^{\otimes n}[C_{\s}]\otimes E_{\ZZ}^{\otimes m}[C_{\t}]
\xrightarrow{\delta^{\ZZ}_{\s}\otimes \delta^{\ZZ}_{\t}}
(E_{\ZZ}^{\otimes n})_{R_{\s}}\otimes (E_{\ZZ}^{\otimes m})_{R_{\t}}
\]
where $\beta(u\otimes \{R_{\add{\s}{\t}}\}) = u\{R_{\add{\s}{\t}}\}[C_{\add{\s}{\t}}]$ for $u\in E_{\ZZ}^{\otimes (n+m)}$, so that $\zeta_{\ZZ}$ is the same as the right `multiplication' by $Y_{\add{\s}{\t}}=\{R_{\add{\s}{\t}}\}[C_{\add{\s}{\t}}] $. 
Let
$
\psi': \mathrm{Im}(\delta_{\add{\s}{\t}}^{\ZZ})\to \mathrm{Im}(\delta_{\s}^{\ZZ})
\otimes \mathrm{Im}(\delta_{\t}^{\ZZ})
$
be the map $(\delta^{\ZZ}_{\s} \otimes \delta^{\ZZ}_{\t}) \circ \beta$ restricted to $\mathrm{Im}(\delta_{\add{\s}{\t}}^{\ZZ})$.  Then $\psi_{\ZZ} = \psi' \circ \delta_{\add{\s}{\t}}^{\ZZ}$, and so $\psi'$
satisfies, over $\QQ$,
\begin{align*}
\psi'(\delta_{\add{\s}{\t}}^{\QQ}(v_{\underline{1}^{\add{\s}{\t}}}[C_{\add{\s}{\t}}]))
& =
\psi_{\QQ}(v_{\underline{1}^{\add{\s}{\t}}}[C_{\add{\s}{\t}}])\\
& =
(\delta^{\QQ}_{\s}\otimes \delta^{\QQ}_{\t})(v_{\underline{1}^{\add{\s}{\t}}}[C_{\add{\s}{\t}}] Y_{\add{\s}{\t}})\\
& =\frac{1}{|R_{\add{\s}{\t}}|}
(\delta^{\QQ}_{\s}\otimes \delta^{\QQ}_{\t})(v_{\underline{1}^{\add{\s}{\t}}}
\{R_{\add{\s}{\t}}\}[C_{\add{\s}{\t}}]
Y_{\add{\s}{\t}})\\
&=\frac{\phl^{\lambda+\mu}}{|R_{\add{\s}{\t}}|}(\delta^{\QQ}_{\s}\otimes \delta^{\QQ}_{\t})
(v_{\underline{1}^{\add{\s}{\t}}}\{R_{\add{\s}{\t}}\}[C_{\add{\s}{\t}}])\\
&=
\phl^{\lambda+\mu} (\delta^{\QQ}_{\s}\otimes \delta^{\QQ}_{\t})
(v_{\underline{1}^{\add{\s}{\t}}}[C_{\add{\s}{\t}}])\\
&=
\phl^{\lambda+\mu}( \delta^{\QQ}_{\s}(v_{\underline{1}^{\s}}[C_{\s}])
\otimes \delta^{\QQ}_{\t}(v_{\underline{1}^{\t}}[C_{\t}]) ) \\
&=
\phl^{\lambda+\mu}( \delta^{\ZZ}_{\s}(v_{\underline{1}^{\s}}[C_{\s}])
\otimes \delta^{\ZZ}_{\t}(v_{\underline{1}^{\t}}[C_{\t}]) )
\end{align*}
where the third equality holds since $R_{\add{\s}{\t}}$ is
the stabilizer of $v_{\underline{1}^{\add{\s}{\t}}}$, and
the fourth follows since
$Y_{\add{\s}{\t}}^2 = \phl^{\lambda+\mu} Y_{\add{\s}{\t}}$, see
Subsection \ref{subsec:combinatorics-Young-tableaux}.

Assume that $\phi$ admits a splitting $\xi$ defined over $\ZZ_{(p)}$.
Identify $\nabla_{\ZZ_{(p)}}(\lambda+\mu)\cong \mathrm{Im}
(\delta_{\add{\s}{\t}}^{\ZP})$,
$\nabla_{\ZZ_{(p)}}(\lambda)\cong \mathrm{Im}
(\delta^{\ZP}_{\s})$
and $\nabla_{\ZZ_{(p)}}(\mu)\cong \mathrm{Im}
(\delta^{\ZP}_{\t})$. Under this identification, the splitting map $\xi$ of $\phi$
satisfies $$\xi(\delta^{\ZP}_{\add{\s}{\t}}(v_{\underline{1}^{\add{\s}{\t}}}[C_{\add{\s}{\t}}]))
=\delta^{\ZP}_{\s}(v_{\underline{1}^{\s}}[C_{\s}])\otimes
\delta^{\ZP}_{\t}(v_{\underline{1}^{\t}}[C_{\t}])
= \delta^{\ZZ}_{\s}(v_{\underline{1}^{\s}}[C_{\s}])\otimes
\delta^{\ZZ}_{\t}(v_{\underline{1}^{\t}}[C_{\t}]). $$
Since
$\dim_{\QQ}\Hom_{\mathrm{GL}_{N}(\QQ)}(\nabla_{\QQ}(\lambda+\mu),
\nabla_{\QQ}(\lambda)\otimes \nabla_{\QQ}(\mu))=1$,
we must have $\psi' = \phl^{\lambda+\mu}\xi$.
To proceed, note that
$N\geq n+m$ implies that elements of $\sym{n+m}$ may be viewed as functions from $\{1,\dotsc, n+m\}$ to $\{ 1,\dotsc, N\}$. Let $\underline{\omega}$ be the identity element in $\sym{n+m}$, viewed in this way.  Then, since 
$$
\psi'(\delta_{\add{\s}{\t}}^{\QQ} (v_{\underline{\omega}}[C_{\add{\s}{\t}}])) = (\delta_{\s}^{\QQ} \otimes \delta_{\t}^{\QQ}) (v_{\underline{\omega}}[C_{\add{\s}{\t}}]Y_{\add{\s}{\t}})= (v_{\underline{\omega}} [C_{\add{\s}{\t}}] Y_{\add{\s}{\t}}) \otimes \{ R_{\s} R_{\t^{+n}} \},$$
we have, by Lemmas \ref{lem:basic-gcds}(4) and \ref{lem:lem-for-product-of-Young-symmetrizers}(1),
\begin{align*}
\mathrm{gcd}_{\mathcal{L}_1}(\psi'(\delta^{\QQ}_{\add{\s}{\t}}
(v_{\underline{\omega}}[C_{\add{\s}{\t}}])))
& =\mathrm{gcd}_{\mathcal{L}_2}(v_{\underline{\omega}}
[C_{\add{\s}{\t}}]Y_{\add{\s}{\t}} \otimes \{R_{\s}R_{\t^{+n}}\})\\
& =\mathrm{gcd}_{\mathcal{L}_3}(v_{\underline{\omega}}
[C_{\add{\s}{\t}}]Y_{\add{\s}{\t}} \otimes \{R_{\s}R_{\t^{+n}}\})\\
& =
\mathrm{gcd}_{\ZZ\sym{n+m}\{ R_{\s}R_{\t^{+n}} \}}([C_{\add{\s}{\t}}]Y_{\add{\s}{\t}} \{R_{\s}R_{\t^{+n}}\}) \\
& =
\mathrm{gcd}_{\ZZ\sym{n+m}}([C_{\add{\s}{\t}}]Y_{\add{\s}{\t}} \{R_{\s}R_{\t^{+n}}\}) = \theta_{\lambda,\mu},
\end{align*}
where
$\mathcal{L}_1 := \mathrm{Im}(\delta^{\ZZ}_{\s} \otimes \delta^{\ZZ}_{\t})$ is a $\ZZ$-summand of $\mathcal{L}_2 := (E^{\otimes (n+m)}_{\ZZ})_{R_{\s}R_{\t^{+n}}}$, which also has another $\ZZ$-summand $\mathcal{L}_3 := (\ZZ\mathcal{B}_{\sym{n+m}})_{R_{\s}R_{\t^{+n}}}$ that is $\ZZ$-isomorphic to $\ZZ\sym{n+m}\{ R_{\s}R_{\t^{+n}} \}$, a $\ZZ$-summand of $\ZZ\sym{n+m}$.  Since $\psi' = \phl^{\lambda + \mu} \xi$, and $\xi$ is defined over $\ZP$, we thus have
$$
\frac{\theta_{\lambda,\mu}}{\phl^{\lambda+ \mu}} = \frac{1}{\phl^{\lambda+\mu}} \gcd\nolimits_{\mathcal{L}_1}(\psi'(\delta^{\QQ}_{\add{\s}{\t}}
(v_{\underline{\omega}}[C_{\add{\s}{\t}}]))) = \gcd\nolimits_{\mathcal{L}_1}(\xi(\delta_{\add{\s}{\t}}^{\ZP} (v_{\underline{\omega}}[C_{\add{\s}{\t}}])))
\in \ZP.
$$
In particular, since $\theta_{\lambda,\mu} \mid \phl^{\lambda+\mu}$ by Lemma \ref{lem:lem-for-product-of-Young-symmetrizers}(3), we must have $p \nmid \frac{\phl^{\lambda+\mu}}{\theta_{\lambda,\mu}}$.

Conversely, suppose that $p\nmid \frac{\phl^{\lambda+\mu}}{\theta_{\lambda,
	\mu}}$, or equivalently, $\frac{\theta_{\lambda,\mu}}{\phl^{\lambda,\mu}} \in \ZP$.  For each $v_f\in \mathcal{B}$, there exists $g_f \in \ZZ G_{\ZZ}$ such that $g_f( v_{\underline{\omega}}) = v_f$, where $\underline{\omega}$ continues to denote the identity of $\sym{n+m}$, viewed as a function from $\{1,\dotsc, n+m \} \to \{1,\dotsc, N\}$.  Thus,
\begin{align*}
\mathrm{gcd}_{\mathcal{L}_1}(\psi'(\delta^{\QQ}_{\add{\s}{\t}}
(v_{f}[C_{\add{\s}{\t}}]))) &= \mathrm{gcd}_{\mathcal{L}_1}(\psi'(\delta^{\QQ}_{\add{\s}{\t}}
(g_f ( v_{\underline{\omega}})[C_{\add{\s}{\t}}]))) \\
&= \mathrm{gcd}_{\mathcal{L}_1}(g_f(\psi'(\delta^{\QQ}_{\add{\s}{\t}}
(v_{\underline{\omega}}[C_{\add{\s}{\t}}])))),
\end{align*}
which is an integer multiple of $\mathrm{gcd}_{\mathcal{L}_1}(\psi'(\delta^{\QQ}_{\add{\s}{\t}}
(v_{\underline{\omega}}[C_{\add{\s}{\t}}]))) = \theta_{\lambda,\mu}$, since $g_f \in \ZZ G_{\ZZ}$ stabilizes $\mathcal{L}_1$.  Consequently,
$$
\mathrm{gcd}_{\mathcal{L}_1}(\frac{\psi'}{\phl^{\lambda+\mu}}(\delta^{\QQ}_{\add{\s}{\t}}
(v_{f}[C_{\add{\s}{\t}}]))) = \frac{\mathrm{gcd}_{\mathcal{L}_1}(\psi'(\delta^{\QQ}_{\add{\s}{\t}}
(v_{f}[C_{\add{\s}{\t}}])))}{\theta_{\lambda,\mu}} \, \frac{\theta_{\lambda,\mu}}{\phl^{\lambda,\mu}} \in \ZZ\frac{\theta_{\lambda,\mu}}{\phl^{\lambda,\mu}} \subseteq \ZP.
$$
Therefore, $\frac{\psi'}{\phl^{\lambda+\mu}}$ is defined over $\ZP$, and, as we've seen above, is a splitting map for $\phi$.
\end{proof}

Without the assumption $N\geq n+m$ in the above theorem, we get the following partial result.

\begin{cor}\label{cor:splitting-unstable-case}
Let $N$ be a natural number, and let $\lambda\vdash n$ and $\mu\vdash m$ be partitions of at most $N$ parts. If $p\nmid
\frac{\phl^{\lambda+\mu}}{\theta_{\lambda,\mu}}$, then the canonical $\mathrm{GL}_N(\FF)$-morphism $\iota_{\lambda,\mu}:\Delta(\lambda+\mu)\to \Delta(\lambda)\otimes\Delta(\mu)$ admits a splitting defined over $\ZZ_{(p)}$.
\end{cor}

\begin{proof} Suppose that $p\nmid \frac{\phl^{\lambda+\mu}}{\theta_{\lambda,\mu}}$. If $N\geq n+m$, then the statement follows from Theorem \ref{thm:split-condition-main}. If $N<n+m$, then let $N'=n+m$ and, by Theorem \ref{thm:split-condition-main}, the canonical $\mathrm{GL}_{N'}(
	\FF)$-morphism $\Delta_{N'}(\lambda+\mu)\to \Delta_{N'}(\lambda)\otimes \Delta_{N'}(\mu)$ admits a splitting defined over $\ZZ_{(p)}$ (here we put the subscript $N'$ to emphasize that $\Delta_{N'}(\lambda), \Delta_{N'}(\mu)$ and $\Delta_{N'}(\lambda+\mu)$ are Weyl modules for
	$\mathrm{GL}_{N'}(\FF)$). Now applying the Schur functor $d_{N',N}$ from polynomial representations of $\mathrm{GL}_{N'}(\FF)$ to polynomial representations of $\mathrm{GL}_N(\FF)$ (see \cite[\S6.5]{Green80}), we deduce that the canonical $\mathrm{GL}_N(\FF)$-morphism $\Delta(\lambda+\mu)\to
	\Delta(\lambda)\otimes \Delta(\mu)$ also admits a splitting defined over $\ZZ_{(p)}$.
\end{proof}

\begin{eg}
The following examples in type $A$ show that Andersen's condition is generally not sufficient for $\iota_{\lambda,\mu}$ to split (over $
\ZP$), and in the case $N<n+m$, the condition $p\nmid \frac{\phl^{\lambda+\mu}}{\theta_{\lambda,\mu}}$ is not always necessary.
Note that 
for a polynomial weight $(\lambda_1,\ldots, \lambda_r )$  (with $r\leq N$) of $\mathrm{GL}_N(\FF)$, its associated dominant integral weight for $\mathrm{SL}_N(\FF)$ is $(\lambda_1-\lambda_2) \omega_1 + \dotsb + (\lambda_{N-1}-\lambda_N)\omega_{N-1}$, where $\omega_1,\dotsc, \omega_{N-1}$ are the fundamental weights of $\mathrm{SL}_N(\FF)$ and $\lambda_i$ is set to be zero if $i>r$.

\begin{enumerate}
\item $\lambda= (1,1)$, $\mu=(1,0)$:  In this case, Andersen's condition is empty (regardless of $N$), while $p\nmid \frac{\phl^{\lambda+\mu}}{\theta_{\lambda,\mu}}$ yields $p \ne 3$.  When $N = 2$, $\Delta(\lambda)$ is the one-dimensional determinant representation for  $\mathrm{GL}_2(\FF)$ and so, $\iota_{\lambda,\mu}: \Delta(\lambda + \mu) \to \Delta(\lambda) \otimes \Delta(\mu)$ is the canonical isomorphism, and thus splits over $\ZP$ for all $p$, showing that the condition $p \ne 3$ is not necessary.  If $N\geq 3$, $\iota_{\lambda,\mu}$ splits over $\ZZ_{(p)}$ if and only if $p\neq 3$, as predicted by Theorem \ref{thm:split-condition-main}, showing that Andersen's condition is insufficient.

\item $\lambda=(1,1)$, $\mu=(4,2)$: In this case, Andersen's condition is empty for $N = 2$, and is $p \ne 3$ for $N \geq 3$, while $p\nmid \frac{\phl^{\lambda+\mu}}{\theta_{\lambda,\mu}}$ yields $p \notin \{2,3\}$.  Once again, when $N=2$, $\Delta(\lambda)$ is the one-dimensional determinant representation for  $\mathrm{GL}_2(\FF)$, and so $\iota_{\lambda,\mu}$ splits over $\ZP$ for all $p$, show that the condition $p \notin \{2,3\}$ is unnecessary.  For $N \geq 3$, we found by brute force that $\iota_{\lambda,\mu}$ splits over $\ZP$ if and only if $p \notin \{2,3\}$, showing once again that Andersen's condition is insufficient.  
\end{enumerate}
\end{eg}

We also have a result analogous to Theorem \ref{thm:split-condition-main} for the symmetric groups.

\begin{prop} \label{prop:split-condition-Specht}
	Let $\lambda,\mu$ be partitions of $n$ and $m$ respectively. Then the canonical $\FF\sym{n+m}$-morphism $\jmath_{\lambda,\mu}:S^{\FF}_{\lambda+\mu}\to \mathrm{Ind}_{\sym{n}\times \sym{m}}^{\sym{n+m}}(S^{\FF}_\lambda\boxtimes S^{\FF}_\mu)$ admits a splitting defined over $\ZZ_{(p)}$ if and only if $p\nmid \frac{\phl^{\lambda+\mu}}{\theta_{\lambda,\mu}}$.  Here, we identify $\sym{n} \times \sym{m}$ with the subgroup $\sym{n}\sym{m}^{+n}$ of $\sym{n+m}$.
\end{prop}

\begin{proof}
Let $\s \in \Tab(\lambda)$ and $\t \in \Tab(\mu)$.  Recall that the dual Specht modules $S^{\mathcal{O}}_{\lambda}$ and $S^{\mathcal{O}}_{\mu}$ (over any commutative ground ring $\mathcal{O}$ with $1$) may be realized as left ideals of the symmetric group rings generated by the Young symmetrizers $Y_{\s}$ and $Y_{\t}$ respectively.  The induced module $\mathrm{Ind}_{\sym{n}\times \sym{m}}^{\sym{n+m}}(S^{\mathcal{O}}_\lambda\boxtimes S^{\mathcal{O}}_\mu)$ may then be realized as the left ideal generated by $Y_{\s}Y_{\t^{+n}}$. As shown in the proof of Lemma \ref{lem:lem-for-product-of-Young-symmetrizers}, \[Y_{\add{\s}{\t}}=\{\Gamma\}Y_{\s}Y_{\t^{+n}} \in \mathcal{O}\sym{n+m} Y_{\s}Y_{\t^{+n}}\] 
	where $\Gamma$ is a left transversal of $R_{\s}R_{\t^{+n}}$ in $R_{\add{\s}{\t}}$.  As such, the left ideal in $\mathcal{O}\sym{n+m}$ generated by $Y_{\s}Y_{\t^{+n}}$, $\mathrm{Ind}_{\sym{n}\times \sym{m}}^{\sym{n+m}}(S^{\mathcal{O}}_\lambda\boxtimes S^{\mathcal{O}}_\mu)$, contains $S^{\mathcal{O}}_{\lambda+\mu}$, which is generated by $Y_{\add{\s}{\t}}$.  Under this realisation, the canonical ${\mathcal{O}}\sym{n+m}$-morphism $\jmath_{\lambda,\mu}:S^{\mathcal{O}}_{\lambda+\mu}\to \mathrm{Ind}_{\sym{n}\times \sym{m}}^{\sym{n+m}}(S^{\mathcal{O}}_\lambda\boxtimes S^{\mathcal{O}}_\mu)$ is the inclusion map.
	
	Let $\chi: \mathrm{Ind}^{\sym{n+m}}_{\sym{n}\times \sym{m}}(S^{\QQ}_\lambda\boxtimes S^{\QQ}_\mu) \to S^{\QQ}_{\lambda+\mu} $ be a nonzero $\QQ\sym{n+m}$-morphism, which exists since $\QQ\sym{n+m}$ is semisimple. Then $\chi(Y_{\s}Y_{\t^{+n}}) = \alpha Y_{\add{\s}{\t}}$ for some $\alpha \in \QQ \sym{n+m}$.  We evaluate $\chi( Y_{\s}Y_{\t^{+n}} (Y_{\s} Y_{\t^{+n}}))$ in two ways.  Firstly,
	\begin{align*}
	\chi( Y_{\s}Y_{\t^{+n}} (Y_{\s} Y_{\t^{+n}})) = \chi((Y_{\s})^2 (Y_{\t^{+n}})^2))
	&= \phl^{\lambda}\phl^{\mu} \chi (Y_{\s}Y_{\t^{+n}})
	= \phl^{\lambda}\phl^{\mu}( \alpha Y_{\add{\s}{\t}}).
	\end{align*}
	On the other hand, we also have
	\begin{align*}
	\chi( Y_{\s}Y_{\t^{+n}}(Y_{\s} Y_{\t^{+n}})) &= Y_{\s}Y_{\t^{+n}} \chi (Y_{\s} Y_{\t^{+n}}) \\
	&= Y_{\s}Y_{\t^{+n}} \alpha Y_{\add{\s}{\t}} \\
	&= \{R_{\s}\}\{R_{\t^{+n}}\}[C_{\s}][C_{\t^{+n}}]\alpha \{R_{\add{\s}{\t}}\}[C_{\add{\s}{\t}}].
	\end{align*}
	Since $[C_{\s}][C_{\t^{+n}}]\sigma \{R_{\add{\s}{\t}}\} = [C_{\add{\s}{\t}}]\sigma \{R_{\add{\s}{\t}}\} \in \{0, \pm [C_{\add{\s}{\t}}]\{R_{\add{\s}{\t}}\} \}$ for all $\sigma \in \sym{n+m}$ by \cite[Lemma 4.6]{James78}, we have $[C_{\s}][C_{\t^{+n}}]\alpha \{R_{\add{\s}{\t}}\} = z[C_{\s}][C_{\t^{+n}}]\{R_{\add{\s}{\t}}\}$ for some $z \in \QQ$, and hence, continuing with the second evaluation, we get
	$$
	\chi( Y_{\s}Y_{\t^{+n}}(Y_{\s} Y_{\t^{+n}})) = zY_{\s}Y_{\t^{+n}}Y_{\add{\s}{\t}}.$$
	Equating the two evaluations, we have, for $y := z/(\phl^{\lambda}\phl^{\mu}) \in \QQ$,
	$$ y Y_{\s}Y_{\t^{+n}}Y_{\add{\s}{\t}} = \alpha Y_{\add{\s}{\t}} = \chi(Y_{\s}Y_{\t^{+n}}).$$
	Thus,
	$$\chi(Y_{\add{\s}{\t}}) = \chi( \{\Gamma\} Y_{\s}Y_{\t^{+n}}) = \{\Gamma\} \chi(Y_{\s}Y_{\t^{+n} }) = \{ \Gamma \} yY_{\s}Y_{\t^{+n}} Y_{\add{\s}{\t}} = y (Y_{\add{\s}{\t}})^2 = y\,\phl^{\lambda+\mu} Y_{\add{\s}{\t}}.$$
	Now, $\chi$ is a splitting map for the inclusion map $\jmath_{\lambda,\mu}$ if, and only if,
	$\chi(Y_{\add{\s}{\t}}) = Y_{\add{\s}{\t}}$, i.e.\ $y = \frac{1}{\phl^{\lambda+\mu}}$.  Now this splitting map is defined over $\ZP$ if and only if $$S^{\ZP}_{\lambda+\mu} = \ZP \sym{n+m} Y_{\add{\s}{\t}} \ni \chi(Y_{\s} Y_{\t^{+n}}) = \frac{1}{\phl^{\lambda+\mu}} Y_{\s}Y_{\t^{+n}}Y_{\add{\s}{\t}},$$ i.e.\ $\ZP \ni \gcd_{\ZZ\sym{n+m}Y_{\add{\s}{\t}}}( \frac{1}{\phl^{\lambda+\mu}} Y_{\s}Y_{\t^{+n}}Y_{\add{\s}{\t}})= \frac{1}{\phl^{\lambda+\mu}} \gcd_{\ZZ\sym{n+m}}(Y_{\s}Y_{\t^{+n}}Y_{\add{\s}{\t}}) = \frac{\theta_{\lambda,\mu}}{\phl^{\lambda+\mu}}$, or equivalently, $p \nmid \frac{\phl^{\lambda+\mu}}{\theta_{\lambda,\mu}}$; here, we have used the fact that $\ZZ\sym{n+m}Y_{\add{\s}{\t}}$ is a $\ZZ$-summand of $\ZZ\sym{n+m}$ and Lemmas \ref{lem:basic-gcds}(2,4) and \ref{lem:lem-for-product-of-Young-symmetrizers}(1).
\end{proof}

\begin{rem}
Theorem \ref{thm:split-condition-main} and Proposition \ref{prop:split-condition-Specht} appears to be intimately related, but we are unable to immediately prove one using the other.  To be sure, one can certainly apply the Schur functor to Theorem \ref{thm:split-condition-main} to deduce immediately the reverse direction of Proposition \ref{prop:split-condition-Specht}, but the forward direction does not seem to be obtainable in this way due to the lack of inverse Schur functor in general.
More precisely, the Hom spaces $\Hom_{\mathrm{GL}_N(\FF)}(X,Y)$ and $\Hom_{\FF \sym{n+m}}(fX,fY)$ for polynomial $\mathrm{GL}_N(\FF)$-modules $X$ and $Y$ are not necessarily isomorphic, even when both $X$ and $Y$ are filtered by Weyl modules (for example, $p=2$, $X= \Delta(3,1^4)$ and $Y = \Delta(5,1^2)$).  We note however that, in the case of $p \geq 5$, these Hom spaces are indeed isomorphic \cite[Corollary 3.9.2]{HN04}, and using this, one can show that Theorem \ref{thm:split-condition-main} and Proposition \ref{prop:split-condition-Specht} are equivalent in this case.
\end{rem}


Our next result relates $\theta_{\lambda,\mu}$ to the denominator of $f_{\add{\IT^{\lambda}}{\IT^{\mu}}}$ when the last column of $[\lambda]$ is no shorter than the first column of $[\mu]$, which is equivalent to either $\add{\IT^{\lambda}}{\IT^{\mu}} \in \STab(\lambda+\mu)$ or $\add{\IT^{\lambda}}{\IT^{\mu}} = \IT_{\lambda+\mu}$.

\begin{thm} \label{thm:denominator}
Let $\lambda \vdash n$ and $\mu \vdash m$, and suppose that the last column of $[\lambda]$ is no shorter than the first column of $[\mu]$.  Then $$\theta_{\lambda,\mu} = \frac{\phl^{\lambda}\phl^{\mu}}{\dd_{\add{\IT^{\lambda}}{\IT^{\mu}}}},$$ where
$\dd_{\add{\IT^{\lambda}}{\IT^{\mu}}}\in \ZZ^+$ is the denominator of $f_{\add{\IT^{\lambda}}{\IT^{\mu}}}$ (see Lemma \ref{lem:denominator-young-basis}).

In particular, $\iota_{\lambda,\mu}$ (when $N\geq n+m$) and $\jmath_{\lambda,\mu}$ admit a splitting defined over $\ZP$ if and only if $p \nmid \frac{\phl^{\lambda+\mu} \dd_{\add{\IT^{\lambda}}{\IT^{\mu}}}}{\phl^{\lambda} \phl^{\mu}}$.
\end{thm}
	

\begin{proof}
Let $\s = \IT_{\lambda}$ and $\t = \IT_{\mu}$.  Then $\add{\s}{\t} = \IT_{\lambda + \mu}$.
By Theorem \ref{thm:basics-on-young-basis}(5), we have
 \begin{align*}
 Y_{\s} &=
 \gamma_{\IT^{\lambda'}}\sigma_\lambda f_{\IT^\lambda,\IT_\lambda}; \\
 Y_{\t} &=
 \gamma_{\IT^{\mu'}}\sigma_\mu f_{\IT^\mu,\IT_\mu}; \\
 Y_{\add{\s}{\t}} = Y_{\IT_{\lambda+\mu}}&=
 \gamma_{\IT^{(\lambda+\mu)'}}\sigma_{\lambda+\mu} f_{\IT^{\lambda+\mu},\IT_{\lambda+\mu}}.
 \end{align*}
 Thus, we have, by the various parts of Proposition \ref{prop:products-of-young-basis-of-different-size} as indicated below,
 \begin{alignat*}{2}
 Y_{\s}Y_{\t^{+n}}Y_{\add{\s}{\t}}  &= Y_{\t^{+n}} Y_{\s} Y_{\add{\s}{\t}} \\
 &=
 \gamma_{\IT^{\mu'}} \gamma_{\IT^{\lambda'}}\gamma_{\IT^{(\lambda+\mu)'}}  (\sigma_{\mu} f_{\IT^{\mu},\IT_{\mu}})^{+n} \sigma_\lambda
 f_{\IT^\lambda,\IT_\lambda}\sigma_{\lambda+\mu}
 f_{\IT^{\lambda+\mu},\IT_{\lambda+\mu}} \\
 &=\gamma_{\IT^{\mu'}} \gamma_{\IT^{\lambda'}}\gamma_{\IT^{(\lambda+\mu)'}}  (\sigma_{\mu} f_{\IT^{\mu},\IT_{\mu}})^{+n} \sigma_\lambda
 f_{\IT^\lambda,\IT_\lambda} \left(f_{\IT_{\lambda+\mu}, \IT_{\lambda+\mu}} + \sum_{\mathfrak{r} \ne \t_{\lambda+\mu}} a_{\mathfrak{r}}f_{\mathfrak{r},\IT_{\lambda+\mu}}\right) & &\text{(by (3))} \\
 &=\gamma_{\IT^{\mu'}}\gamma_{\IT^{\lambda'}}\gamma_{\IT^{(\lambda+\mu)'}} \gamma_{\IT_\lambda}
 \sigma_\lambda \sigma_{\mu}^{+n} (f_{\IT^{\mu},\IT_{\mu}})^{+n} \left(f_{\add{\IT^{\lambda}}{\IT_{\mu}}, \IT_{\lambda+\mu}} + \sum_{\mathfrak{p} \ne \IT_{\mu}} a_{\add{\IT^{\lambda}}{\mathfrak{p}}}f_{\add{\IT^{\lambda}}{\mathfrak{p}},\IT_{\lambda+\mu}}\right) &\:\:&\text{(by (1))} \\
 &= \gamma_{\IT^{\mu'}}\gamma_{\IT^{\lambda'}}\gamma_{\IT^{(\lambda+\mu)'}} \gamma_{\IT_\lambda} \gamma_{\IT_{\mu}}
 \sigma_\lambda \sigma_{\mu}^{+n} f_{\add{\IT^{\lambda}}{\IT^{\mu}}, \IT_{\lambda+\mu}} \quad \text{(by (2) and Theorem \ref{thm:basics-on-young-basis}(3))} \\
 &= \phl^{\lambda} \phl^{\mu} \sigma_{\lambda}\sigma^{+n}_{\mu} f_{\add{\IT^{\lambda}}{\IT^{\mu}}} \qquad \qquad \qquad \qquad \;\;\: \text{(by Theorem  \ref{thm:basics-on-young-basis}(1))}.
 \end{alignat*}
 Thus, we have, by Lemmas \ref{lem:basic-gcds}(2) and \ref{lem:denominator-young-basis},
 \begin{align*}
 \theta_{\lambda,\mu} = \gcd\nolimits_{\ZZ\sym{n+m}}(Y_{\s}Y_{\t^{+n}}Y_{\add{\s}{\t}}) &= \gcd\nolimits_{\ZZ\sym{n+m}}(\phl^{\lambda} \phl^{\mu} \sigma_{\lambda}\sigma^{+n}_{\mu} f_{\add{\IT^{\lambda}}{\IT^{\mu}}}) \\
 &= \phl^{\lambda} \phl^{\mu} \gcd\nolimits_{\ZZ\sym{n+m}}(f_{\add{\IT^{\lambda}}{\IT^{\mu}}}) = \frac{\phl^{\lambda}\phl^{\mu}}{\dd_{\add{\IT^{\lambda}}{\IT^{\mu}}}}.
 \end{align*}
The last assertion now follows from Theorem \ref{thm:split-condition-main} and Proposition \ref{prop:split-condition-Specht}.
\end{proof}

Since $\dd_{\add{\IT^{\lambda}}{\IT^{\mu}}} \in \ZZ^+$ by Lemma \ref{lem:denominator-young-basis}, we have the following immediately corollary, when combined with Lemma \ref{lem:lem-for-product-of-Young-symmetrizers}(3).

\begin{cor} \label{cor:theta-gcd}
Suppose that the last column of $[\lambda]$ is no shorter than the first column of $[\mu]$.  Then $\theta_{\lambda,\mu} \mid \gcd(\phl^{\lambda}\phl^{\mu},\phl^{\lambda+\mu})$.
\end{cor}

\section{Some examples} \label{sec:examples}

Perhaps to be expected, the computation of $\theta_{\lambda,\mu}$ is difficult in general.  The only work (of which we are aware) that relates to the computation of $Y_{\s}Y_{\t^{+n}}Y_{\add{\s}{\t}}$ is by Raicu \cite[Theorem 1.2]{Raicu14}, in which he provides a simplified way of evaluating $Y_{\s}Y_{\t^{+n}}Y_{\add{\s}{\t}}$ when $\t$ is the unique $(1)$-tableau.  However, it is not clear how one can deduce $\theta_{\lambda,(1)}$ immediately from his result.

In this concluding section, we provide closed formulas for $\theta_{\lambda,\mu}$ for two `easy' cases, in which we compute explicitly the product $Y_{\IT^{\lambda}}Y_{(\IT^{\mu})^{+n}}Y_{\add{\IT^{\lambda}}{\IT^{\mu}}}$ and obtain the greatest common divisor of its coefficients.

We shall use the following notation in this section: for $a,b \in \mathbb{Z}$ with $a\leq b$, we write $[a,b]$ for the integer interval  between $a$ and $b$ (both inclusive), i.e.
$$ [a,b] := \{ c \in \ZZ \mid a \leq c \leq b \}.$$



\subsection{The case $\lambda = (1^n)$ and $\mu = (m)$}

We have $Y_{\IT^{\lambda}} = [ \sym{n} ]$, $Y_{\IT^{\mu}} = \{\sym{m}\}$ and $Y_{\add{\IT^{\lambda}}{\IT^{\mu}}} = \{\sym{\{1\} \cup [n+1,n+m]}\}[\sym{n}]$.
Let $\gamma_0 = 1_{\sym{n+m}}$, the identity of $\sym{n+m}$, and let $\gamma_i = (1, n+i)$ for all $1\leq i \leq m$.
Then $\Gamma := \{ \gamma_0,\dotsc, \gamma_m\}$ is a transversal of $\sym{m}^{+n}$ in $\sym{\{1\} \cup [n+1,n+m]}$.  Thus,
\begin{align*}
Y_{\IT^{\lambda}} Y_{(\IT^{\mu})^{+n}} Y_{\add{\IT^{\lambda}}{\IT^{\mu}}} &= [ \sym{n} ] \{\sym{m}^{+n}\}\{\sym{\{1\} \cup [n+1,n+m]}\}[\sym{n}] \\
&= \sum_{\substack{\sigma_1,\sigma_2 \in \sym{n} \\ \tau_1,\tau_2 \in \sym{m}}} \sum_{i=0}^m \sgn(\sigma_1\sigma_2) \sigma_1 \tau_1^{+n} \gamma_i \tau_2^{+n} \sigma_2.
\end{align*}
Write $Y_{\IT^{\lambda}} Y_{(\IT^{\mu})^{+n}} Y_{\add{\IT^{\lambda}}{\IT^{\mu}}} = \sum_{\rho \in \sym{n+m}} c_{\rho} \rho$; then $c_{\rho} = \sum \sgn(\sigma_1\sigma_2)$ where the sum runs over all $\sigma_1,\sigma_2 \in \sym{n}$, $\tau_1,\tau_2 \in\sym{m}$ and $\gamma_i \in \Gamma$ such that $\sigma_1\tau_1^{+n} \gamma_i \tau_2^{+n}\sigma_2 = \rho$.

Fix $i \in [1,m]$, and let $H = \sym{[2, n]}$ and $K_i = \sym{[n+1,n+m] \setminus \{n+i\}}$.  For $a \in [1,n]$ and $b \in [1,m]$, define $\alpha_a = (1,a) $ and $\beta_b = (n+i, n+b)$ (where $(1,1)$ and $(n+i,n+i)$ are to be read as $1_{\sym{n+m}}$).  Then $\{ \alpha_1,\dotsc, \alpha_n \}$ is a transversal of $H$ in $\sym{n}$ while $\{ \beta_1,\dotsc, \beta_m \}$ is a transversal of $K_i$ in $\sym{m}^{+n}$, so that
$$
\sym{n}\sym{m}^{+n} \gamma_i \sym{m}^{+n} \sym{n} = \bigcup_{a=1}^n \bigcup_{b=1}^m \sym{n}\beta_b K_i \gamma_i \sym{m}^{+n} H \alpha_a.
$$
Furthermore, we clearly have
$$
C_{a,b} := \sym{n}\beta_b K_i \gamma_i \sym{m}^{+n} H \alpha_a \subseteq \{ \rho \in \sym{n+m} \mid \rho([1,n]\setminus \{a\}) \subseteq [1,n],\ \rho(a) =n+b \};
$$
in particular, the $C_{a,b}$'s are pairwise disjoint. We claim that the above inequality is in fact an equality. If $\rho \in \sym{n+m}$ such that $\rho(j) \in [1,n]$ for all $j \in [1,n] \setminus \{a\}$, and $\rho(a) = n+b$, let $a'\in [1,n]$ be the unique element such that $a' \notin \rho([1,n])$.  Then $\gamma_i \beta_b \alpha_{a'} \rho \alpha_a \in\sym{n+m}$ sends $[1,n]$ to $[1,n]$, and fixes $1$, so that $\gamma_i \beta_b \alpha_{a'} \rho \alpha_a =  \tau_{\rho,i}^{+n}h_{\rho,i}$ for some unique $h_{\rho,i} \in H$ and $\tau_{\rho,i} \in \sym{m}$.  Thus, $\rho = \alpha_{a'} \beta_b\gamma_i  \tau_{\rho,i}^{+n} h_{\rho,i}\alpha_a \in \sym{n}\beta_b K_i \gamma_i \sym{m}^{+n} H \alpha_a$. This proves the claim.
In particular, this justifies our notation $C_{a,b}$ which is independent of $i\in [1,m]$.
Now, let $\rho \in C_{a,b}$.  Then we have seen above that there exist unique $a' \in [1,n]$, $h_{\rho,i} \in H$ and $\tau_{\rho,i} \in \sym{m}$ such that $\rho = \alpha_{a'} \beta_b\gamma_i  \tau_{\rho,i}^{+n} h_{\rho,i}\alpha_a$.  For any $\sigma \in \sym{n}$, $\kappa \in K_i$, $\tau \in \sym{m}$ and $h \in H$, we have
\begin{align*}
\alpha_{a'} \beta_b\gamma_i  \tau_{\rho,i}^{+n} h_{\rho,i}\alpha_a = \rho = \sigma \beta_b \kappa \gamma_i \tau^{+n} h \alpha_a &\iff \alpha_{a'} \gamma_i  \tau_{\rho,i}^{+n} h_{\rho,i} = \sigma \kappa \gamma_i \tau^{+n} h \\
&\iff \alpha_{a'} h_{\rho,i} \gamma_i  \tau_{\rho,i}^{+n} = \sigma h \gamma_i \kappa \tau^{+n} \\
&\iff (\sigma h)^{-1}\alpha_{a'}h_{\rho,i} = \gamma_i (\kappa (\tau\tau_{\rho,i}^{-1})^{+n}) \gamma_i^{-1} \\
&\iff (\sigma h)^{-1}\alpha_{a'}h_{\rho,i} = \gamma_i (\kappa (\tau\tau_{\rho,i}^{-1})^{+n}) \gamma_i^{-1} = 1_{\sym{n+m}} \\
&\iff \sigma = \alpha_{a'} h_{\rho,i}h^{-1} \text{ and } \tau^{+n} = \kappa \tau_{\rho,i}^{+n},
\end{align*}
where the penultimate line holds because the lefthand side of the previous line has support a subset of $[1,n]$ while the righthand side has support a subset of $\{1 \} \cup [n+1,n+m]$.  Thus, exactly $|K_i||H| = (m-1)!(n-1)!$ such quadruples will contribute to $c_\rho$, with each contributing $\sgn(\sigma h\alpha_a) = \sgn(\alpha_{a'}h_{\rho,i} \alpha_{a}) = \sgn(h_{\rho,i})$.

For $i,j \in [1,m]$ and $\rho \in C_{a,b}$, we have $\alpha_{a'} \beta_b\gamma_i  \tau_{\rho,i}^{+n} h_{\rho,i}\alpha_a = \rho = \alpha_{a'} \beta_b\gamma_j  \tau_{\rho,j}^{+n} h_{\rho,j}\alpha_a$, giving $h_{\rho,i} h_{\rho,j}^{-1} = (\tau_{\rho,i}^{-1})^{+n} \gamma_i \gamma_j \tau_{\rho,j}^{+n}$, which has to be the identity since the two expressions have disjoint support.  Thus, $h_{\rho,i} = h_{\rho,j}$, which we shall now denote as $h_{\rho}$.  Since, for each $i \in [1,m]$, we have exactly $(m-1)!(n-1)!$ contributions to $c_{\rho}$, with each contributing $\sgn(h_{\rho,i}) = \sgn(h_{\rho})$, and since $\sym{n}\sym{m}^{+n} \gamma_0 \sym{m}^{+n}\sym{n} = \sym{n}\sym{m}^{+n}$ is disjoint from the $C_{a,b}$'s, we conclude that $c_{\rho} = m!(n-1)! \sgn(h_{\rho})$.

Finally, for $\rho \in \sym{n}\sym{m}^{+n}$, and $\sigma_1,\sigma_2 \in \sym{n}$ and $\tau_1,\tau_2 \in \sym{m}$, we have
\begin{align*}
\rho = \sigma_1 \tau_1^{+n} \tau_2^{+n} \sigma_2 \iff \sigma_1 \sigma_2 = \sigma_{\rho} \text{ and } \tau_1 \tau_2 = \tau_{\rho},
\end{align*}
where $\sigma_{\rho} \in \sym{n}$ and $\tau_{\rho} \in \sym{m}$ are the unique elements such that $\rho = \sigma_{\rho}\tau_{\rho}^{+n}$.  Thus, there are exactly $|\sym{n}||\sym{m}| = n!m!$ contributions to $c_{\rho}$, with each contributing $\sgn(\sigma_1\sigma_2) = \sgn(\sigma_{\rho})$.

We have therefore shown that:

\begin{prop} \label{prop:(1^n)(m)}
Let $\lambda = (1^n)$ and $\mu = (m)$.  Then
$$
Y_{\IT^{\lambda}} Y_{(\IT^{\mu})^{+n}} Y_{\add{\IT^{\lambda}}{\IT^{\mu}}} =
\sum_{\rho \in \sym{n}\sym{m}^{+n}} m!n!\sgn(\sigma_{\rho}) \rho + \sum_{a=1}^n \sum_{b=1}^m \sum_{\rho \in C_{a,b}} m!(n-1)!\sgn(h_{\rho}) \rho.
$$
In particular, $\theta_{(1^n),(m)} = m!(n-1)!$.
\end{prop}

We thus conclude from Proposition \ref{prop:(1^n)(m)}, Theorem \ref{thm:split-condition-main} and Proposition \ref{prop:split-condition-Specht} that $\iota_{(1^n),(m)}$ (when $N\geq n+m$) and $\jmath_{(1^n),(m)}$ admit a splitting defined over $\ZP$ if and only if $p \nmid \frac{\phl^{(m+1,1^{n-1})}}{m!(n-1)!} = n+m$.

\begin{rem}
It is perhaps not surprising that $p \nmid (n+m)$ is sufficient for the splitting of $\iota_{(1^n),(m)}$, as one may draw the same conclusion by considering the Weyl filtration of $\Delta(\lambda) \otimes \Delta(\mu)$ obtained from the Littlewood-Richardson's rule and the block(s) in which the Weyl factors, in particular $\Delta(\lambda+\mu)$, lie.  Our results show that the condition is also necessary, which agrees with \cite[Theorem 1.4(i)(c)]{GLOW17}.
\end{rem}

\subsection{The case $\lambda = (k,\ell)$ and $\mu = (m)$}

Let $\s = \IT^{\lambda}$ and $\t = \IT^{\mu}$, and let
$$
A_1 := [1,k]; \qquad A_2 := [k+1,k+\ell]; \qquad A_3 := [k+\ell+1,k+\ell+m].$$
Then $C_{\add{\s}{\t}} = C_{\s} = \left< (j-k,j) \mid j \in A_2 \right>$, an elementary Abelian $2$-group of rank $\ell$, which implies in particular that $\kappa = \kappa^{-1}$ for all $\kappa \in C_{\s}$, a fact that we shall use repeatedly in what follows without further comment, while $R_{\add{\s}{\t}} = \sym{A_1 \cup A_3}\sym{A_2}$.  Denote the subgroup $R_{\s}R_{\t^{+(k+\ell)}} = \sym{k}\sym{\ell}^{+k}\sym{m}^{+(k+\ell)} =\sym{A_1} \sym{A_2}\sym{A_3}$ by $R_{\s \times \t}$. Then $$Y_{\s}Y_{\t^{+(k+\ell)}} Y_{\add{\s}{\t}} = \{R_{\s}\}[C_{\s}] \{R_{\t^{+(k+\ell)}} \} [C_{\t^{+(k+\ell)}}]\{ R_{\add{\s}{\t}} \} [C_{\add{\s}{\t}}] = \{ R_{\s \times \t} \} [C_{\s}] \{ R_{\add{\s}{\t}} \}[C_{\s}].$$
Let $\m : R_{\s\times\t} \times C_{\s} \times R_{\add{\s}{\t}} \times C_{\s} \to \sym{k+\ell+m}$ be defined by $\m (\rho, \kappa, \rho', \kappa') = \rho \kappa \rho'\kappa'$.  Then the image of $\m$ is precisely $R_{\s\times\t}C_{\s}R_{\add{\s}{\t}}C_{\s}$, and so $Y_{\s}Y_{\t^{+(k+\ell)}} Y_{\add{\s}{\t}} = \sum_{\sigma \in R_{\s\times \t}C_{\s}R_{\add{\s}{\t}}C_{\s}} a_{\sigma}\sigma$, where
$$
a_{\sigma} = \sum_{(\rho, \kappa, \rho', \kappa') \in \m^{-1}(\{\sigma\})} \sgn(\kappa\kappa').
$$

To continue, we will describe $R_{\s\times\t}C_{\s}R_{\add{\s}{\t}}C_{\s}$ first, then $\m^{-1}(\{\sigma\})$ for $\sigma \in R_{\s\times\t}C_{\s}R_{\add{\s}{\t}}C_{\s}$, and finally compute $a_{\sigma}$.

Denote the conjugacy class of $\sym{k+\ell+m}$ consisting of all permutations having cycle type $(2^\ell, 1^{k-\ell+m})$ by $\C$.
There is a unique element of $C_{\s}$ in $\C$, namely $\pi_{\s} := \prod_{j\in A_2} (j-k, j)$.
The symmetric group $\sym{k+\ell+m}$ acts naturally and transitively (from the left) on $\C$ by conjugation, i.e.\ $g \cdot \pi = g \pi g^{-1}$ for $g \in \sym{k+\ell+m}$ and $\pi \in \C$.  Under this action, the stabilizer of $\pi_{\s}$ is its centralizer in $\sym{k+\ell+m}$, namely $\Delta_2(\sym{\ell})\sym{[\ell+1, k]\cup A_3} C_{\s}$,  where
$$
\Delta_2(\sym{\ell}) = \{ \sigma \sigma^{+k} \mid \sigma \in \sym{\ell} \}.$$
Observe also that $R_{\add{\s}{\t}} \cdot \pi_{\s} =\{ \prod_{j\in A_2} (a_j, j) \mid \text{$a_j$'s distinct elements of $A_1\cup A_3$}\}$.


\begin{prop} \label{P:CRCR}
Let $\sigma \in \sym{k+\ell+m}$ and $\kappa \in C_{\s}$.  Then $\sigma\in R_{\s\times \t} C_{\s} R_{\add{\s}{\t}} \kappa$ if and only if
$\sigma\kappa (A_2) \subseteq A_1 \cup A_2$.
\end{prop}

\begin{proof}
The forward direction of the statement is clear.

Conversely, if $\sigma\kappa (A_2) \subseteq A_1 \cup A_2$, then we can find $\rho \in \sym{A_1}$ such that $\rho(\sigma\kappa(A_2)) \subseteq [1,\ell] \cup A_2$.  Let $\tau= (\rho\sigma \kappa)^{-1}$.  Then $\tau([1,\ell] \cup A_2)$ contains precisely $A_2$ and $\ell$ other integers in $A_1\cup A_3$.  Consider $\tau \cdot \pi_{\s} = \prod_{j\in A_2} (\tau(j-k), \tau(j))$.  Let
\begin{align*}
J &:= \{ j \in A_2 \mid \tau(j-k), \tau(j) \in A_2 \},\\
J' &:= \{ j \in A_2 \mid \tau(j-k), \tau(j) \notin A_2 \}.
\end{align*}
Then $|J| = |J'| =:r$.  Let $J = \{j_1,\dotsc, j_r\}$ and $J'= \{j'_1,\dotsc, j'_r \}$, and let $\rho_1 = \prod_{i=1}^r (j_i, j'_i) \in \sym{A_2}$. Then for each $j \in A_2$,  exactly one of $(\tau\rho_1)(j-k)$ and $(\tau\rho_1)(j)$ lie in $A_2$, so that
$$(\tau\rho_1) \cdot \pi_{\s} = \prod_{j\in A_2} ((\tau\rho_1)(j-k), (\tau\rho_1)(j)) \in R_{\add{\s}{\t}} \cdot \pi_{\s},$$
and hence $(\sigma\kappa)^{-1}\rho^{-1} \rho_1 = \tau \rho_1 \in R_{\add{\s}{\t}}\mathrm{Stab}_{\sym{k+\ell+m}}(\pi_{\s}) = R_{\add{\s}{\t}} C_{\s}$.  Consequently, $(\sigma\kappa)^{-1} \in R_{\add{\s}{\t}}C_{\s}R_{\s\times \t}$ as $\rho,\rho_1 \in R_{\s\times \t}$, or equivalently, $\sigma\kappa \in R_{\s\times \t}C_{\s}R_{\add{\s}{\t}}$ as desired.
\end{proof}

The following corollary provides a description for the set $R_{\s\times\t}C_{\s}R_{\add{\s}{\t}}C_{\s}$ as promised.

\begin{cor} \label{C:CRCR}
$R_{\s\times \t}C_{\s} R_{\add{\s}{\t}} C_{\s} = \{ \sigma \in \sym{k+\ell+m} \mid  \{\sigma(j-k), \sigma(j)\} \not\subseteq A_3 \ \forall j \in A_2 \}$.
\end{cor}

\begin{proof}
That the lefthand side is a subset of the righthand side follows immediately from Proposition \ref{P:CRCR}.
For the converse, let $\sigma$ be an element of the righthand side, and let $\kappa := \prod_{j \in A_2 \cap \sigma^{-1}(A_3)} (j-k, j) \in C_{\s}$.  Then $\sigma\kappa(j) \notin A_3$ for all $j \in A_2$, and thus $\sigma \in R_{\s \times \t} C_{\s} R_{\add{\s}{\t}}\kappa \subseteq R_{\s \times \t} C_{\s} R_{\add{\s}{\t}}C_{\s}$ by Proposition \ref{P:CRCR}.
\end{proof}

Let $\sigma \in R_{\s\times \t}C_{\s} R_{\add{\s}{\t}} C_{\s}$.  Define
\begin{align*}
J_{\sigma} &= \{ j \in A_2 \mid \sigma(j), \sigma(j-k) \notin A_3 \}; \\
J'_{\sigma} &= \{ j \in A_2 \mid \sigma(j) \in A_3\}.
\end{align*}
Then $\sigma \in R_{\s\times \t}C_{\s} R_{\add{\s}{\t}} \kappa$ for some $\kappa \in C_{\s}$ if and only if $\kappa = \kappa_{\sigma,I} := \prod_{j \in I \cup J'_{\sigma}}  (j-k, j)$ for some $I \subseteq J_{\sigma}$ by Corollary \ref{C:CRCR} and Proposition \ref{P:CRCR}.  Thus, to describe $\m^{-1}(\sigma)$, it suffices to consider $\sigma \kappa_{\sigma, I} \in R_{\s\times \t} C_{\s} R_{\add{\s}{\t}}$ for various subsets $I$ of $J_\sigma$.

For the remainder of this paper, we need the following notation. For $\sigma \in \sym{k+\ell+m}$ and $i,j \in [1,3]$, let
$$
X^{\sigma}_{ij} = \sigma(A_i) \cap A_j.
$$ We record some easy consequences and leave their proofs to the reader as an easy exercise.

\begin{lem} \label{L:easy} \hfill

\begin{enumerate}
\item
\begin{enumerate}
\item If $\rho \in R_{\s\times \t}$, then $X^{\sigma \rho}_{ij} = X^{\sigma}_{ij}$ and $X^{\rho\sigma}_{ij} = \rho(X^{\sigma}_{ij})$ for all $i,j \in [1,3]$.
\item If $\rho' \in R_{\add{\s}{\t}}$, then $X^{\sigma \rho'}_{2j} = X^{\sigma}_{2j}$ and $X^{\rho'\sigma}_{i2} = \rho'(X^{\sigma}_{i2})$ for all $i,j \in [1,3]$.
\end{enumerate}
\item Let $\kappa \in C_{\s}$ and $J \subseteq A_1$.  The following statements are equivalent:
\begin{enumerate}
\item $X^{\kappa}_{21} = J$.
\item $\kappa = \prod_{j \in J} (j, j+k)$.
\item $X^{\kappa}_{12} = \pi_{\s}(J)= J^{+k}$.
\end{enumerate}
In particular, $\kappa(X^{\kappa}_{12}) = X^{\kappa}_{21}$, and
$\sgn(\kappa) = (-1)^{|X^{\kappa}_{12}|} = (-1)^{|X^{\kappa}_{21}|}$.
\end{enumerate}
\end{lem}


\begin{prop} \label{P:RCR}
Let $\tau = \rho \kappa \rho'$, where $\kappa \in C_{\s}$, $\rho \in R_{\s \times \t}$ and $\rho' \in R_{\add{\s}{\t}}$, and let $\rho_1 \in R_{\s \times \t}$.  Then $\tau \in \rho_1 C_{\s} R_{\add{\s}{\t}}$ if and only if $\rho_1^{-1} \rho (X^\kappa_{12}) = (\rho_1^{-1}\rho(X^\kappa_{21}))^{+k}$, in which case there exist unique $\kappa_1\in C_{\s}$ and $\rho_1' \in R_{\add{\s}{\t}}$ such that $\tau = \rho_1 \kappa_1 \rho_1'$.

In particular, $|\{ (\rho_1, \kappa_1,\rho_1') \in R_{\s\times \t} \times C_{\s} \times R_{\add{\s}{\t}} \mid \rho_1 \kappa_1 \rho_1' = \tau \}| = |X^{\tau}_{21}|!(k-|X^{\tau}_{21}|)!\ell!m!$.
\end{prop}
\begin{proof} 

If $\tau = \rho_1 \kappa_1 \rho'_1$ with $\kappa_1 \in C_{\s}$ and $\rho'_1 \in R_{\add{\s}{\t}}$, then
\begin{align*}
\sym{A_1 \cup A_2}\sym{A_3} \ni \kappa_1^{-1}\rho_1^{-1} \rho \kappa =  \rho'_1\rho'^{-1}\in R_{\add{\s}{\t}} = \sym{A_1 \cup A_3}\sym{A_2}
\end{align*}
so that $\kappa_1^{-1}\rho_1^{-1} \rho \kappa \in \sym{A_1 \cup A_2}\sym{A_3} \cap \sym{A_1 \cup A_3}\sym{A_2} = \sym{A_1}\sym{A_2}\sym{A_3}$.  In particular, for all $i \in X^{\kappa}_{21}$, we have $\rho_1^{-1} \rho \kappa (i) \in A_2$ (since $\kappa(i) \in A_2$ and $\rho_1^{-1}\rho \in R_{\s\times \t}$) while $\kappa_1^{-1}(\rho_1^{-1} \rho \kappa (i)) \in A_1$ (since $i\in A_1$), so that $\rho_1^{-1} \rho \kappa (i) \in X^{\kappa_1}_{12}$.
By Lemma \ref{L:easy}, we have $ |X^{\kappa}_{21}| = |X^{\tau}_{21}| = |X^{\kappa_1}_{21}| =|X^{\kappa_1}_{12}|$  and hence $X^{\kappa_1}_{12} = \rho_1^{-1} \rho \kappa (X^{\kappa}_{21}) = \rho_1^{-1} \rho (X^{\kappa}_{12})$; consequently $\kappa_1 = \prod_{j \in X^{\kappa}_{12}} (\rho_1^{-1} \rho (j)-k, \rho_1^{-1} \rho (j))$.
In particular, $\kappa_1$ is unique, and hence so is $\rho_1'$.
Repeating the same argument with $X^{\kappa}_{12}$ replacing $X^{\kappa}_{21}$, we also get $X^{\kappa_1}_{21} = \rho_1^{-1} \rho (X^{\kappa}_{21})$.  Thus
$$\rho_1^{-1}\rho(X^\kappa_{12}) = X^{\kappa_1}_{12} = (X^{\kappa_1}_{21})^{+k} = (\rho_1^{-1}\rho(X^\kappa_{21}))^{+k}.$$

Conversely, if $\rho_1^{-1} \rho (X^\kappa_{12}) = (\rho_1^{-1}\rho(X^\kappa_{21}))^{+k}$, let
$$\kappa_1 = \prod_{i \in X^{\kappa}_{21}} (\rho_1^{-1}\rho(i), \rho_1^{-1}\rho(i)+k) = \prod_{j \in X^{\kappa}_{12}} (\rho_1^{-1}\rho(j)-k, \rho_1^{-1}\rho(j)) .$$
Then for $j \in X^{\kappa}_{12}$, we have $\kappa (j) \in X^{\kappa}_{21}$ by Lemma \ref{L:easy}(2), so that $ \kappa_1\rho_1^{-1}\rho \kappa(j) = \rho_1^{-1}\rho \kappa(j) + k \in A_2$.  On the other hand, if $j \in A_2 \setminus X^{\kappa}_{12}$, then $\kappa_1\rho_1^{-1}\rho \kappa(j) = \kappa_1(\rho_1^{-1}\rho(j)) = \rho_1^{-1}\rho(j) \in A_2$.  Thus, $\kappa_1\rho_1^{-1}\rho \kappa(A_2) = A_2$, and so $\kappa_1\rho_1^{-1}\rho \kappa \in \sym{A_1}\sym{A_2}\sym{A_3}$, and hence $\kappa_1\rho_1^{-1}\tau = \kappa_1\rho_1^{-1}\rho \kappa \rho' \in R_{\add{\s}{\t}}$ as desired.

Let $B_1=\rho(X^\kappa_{21}) \subseteq A_1$ and $B_2=\rho(X^\kappa_{12}) \subseteq A_2$.  Then $|B_1| = |B_2| = |X^{\kappa}_{21}| = |X^{\tau}_{21}|$ by Lemma \ref{L:easy}. Suppose that $\rho_1^{-1}=\sigma_1\sigma_2\sigma_3$ where $\sigma_i\in\sym{A_i}$ for all $i$. The condition $\rho_1^{-1} \rho (X^\kappa_{12}) = (\rho_1^{-1}\rho(X^\kappa_{21}))^{+k}$ is equivalent to $\sigma_2(B_2)=(\sigma_1(B_1))^{+k}$. For each pair $(\sigma_2,\sigma_3) \in \sym{A_2} \times \sym{A_3}$, $|\{ \sigma_1 \in A_1 \mid (\sigma_1(B_1))^{+k} = \sigma_2(B_2) \}| = |X^{\tau}_{21}|!(k-|X^{\tau}_{21}|)!$.  The last assertion thus follows.
\end{proof}

\begin{lem} \label{L:a_sigma=0}
Let $\sigma \in R_{\s\times\t}C_{\s}R_{\add{\s}{\t}}C_{\s}$.  If $(\sigma(j-k), \sigma(j)) \in R_{\s\times \t}$ for some $j \in A_2$, then $a_{\sigma} = 0$.
\end{lem}

\begin{proof}
The map $f : \m^{-1}(\{ \sigma \}) \to \m^{-1}(\{\sigma \})$ defined by
$$(\rho, \kappa, \rho',\kappa') \mapsto ((\sigma(j-k), \sigma(j))\rho, \kappa, \rho', \kappa'(j-k,j))$$ is a well-defined fixed-point-free involution on $\m^{-1}(\{ \sigma \})$.  Furthermore, $\sgn(\kappa)\sgn(\kappa') = -\sgn(\kappa)\sgn(\kappa'(j-k,j))$, so that the contributions by $(\kappa,\rho,\kappa', \rho')$ and $f(\kappa,\rho,\kappa', \rho')$ to $a_\sigma$ cancel out. Thus $a_{\sigma} = 0$.
\end{proof}

\begin{prop} \label{P:equalsign}
Let $\sigma \in R_{\s \times \t}C_{\s}R_{\add{\s}{\t}}C_{\s}$ such that $(\sigma(j-k), \sigma(j)) \notin R_{\s\times \t}$ for all $j \in A_2$.
Then there exists $\varepsilon_{\sigma} \in \{ \pm 1\}$ such that
$\sgn(\kappa\kappa') =  \varepsilon_\sigma$ for all $(\rho,\kappa,  \rho',\kappa') \in \m^{-1}(\{\sigma\})$.
\end{prop}

\begin{proof}
We assume first that $X^{\sigma}_{21}  = \sigma(A_2) \cap A_1 = \varnothing$.
Let $(\rho, \kappa, \rho',\kappa') \in \m^{-1}(\{\sigma\})$.  Then $
\sigma\kappa' = \rho \kappa \rho'$, and so  $X^{\sigma\kappa'}_{21} = X^{\rho\kappa\rho'}_{21} = \rho(X^{\kappa}_{21})$ by Lemma \ref{L:easy}(1).
On the other hand, by Lemma \ref{L:easy}(2),
\begin{align*}
X^{\sigma\kappa'}_{21} = (\sigma\kappa' ) (A_2) \cap A_1 &= (\sigma\kappa' ) (X^{\kappa'}_{12} \cup (A_2 \setminus X^{\kappa'}_{12})) \cap A_1 \\
&=  (\sigma(X^{\kappa'}_{21}) \cup \sigma(A_2 \setminus X^{\kappa'}_{12})) \cap A_1 = \sigma(X^{\kappa'}_{21}) \cap A_1,
\end{align*}
since $ \sigma(A_2) \cap A_1 = \varnothing$.

We claim that $\sigma(X^{\kappa'}_{21}) \setminus X^{\sigma \kappa'}_{21} = \sigma([1,\ell]) \cap A_2$.
Firstly, for $i \in X^{\kappa'}_{21}$, we have $i \in [1,\ell]$, and $\sigma(i) = \sigma\kappa'(\kappa'(i)) \in A_1 \cup A_2$ by Proposition \ref{P:CRCR}, since $\kappa'(i) \in A_2$.
If $\sigma(i) \in A_1$, then $\sigma(i) \in \sigma(X^{\kappa'}_{21}) \cap A_1 = X^{\sigma\kappa'}_{21}$.
Thus, for $i\in X^{\kappa'}_{21}$ such that $\sigma(i)\not\in X^{\sigma\kappa'}_{21}$, we have $\sigma(i)\in A_2$ and hence $\sigma(i) \in \sigma([1,\ell]) \cap A_2$.
Conversely, if $i' \in [1,\ell] \cap \sigma^{-1}(A_2)$, then $\sigma(i') \in A_2$.
Since $\sigma(A_2) \cap A_1 = \varnothing$ (our assumption) and $(\sigma(i') , \sigma(i'+k)) \notin R_{\s\times \t}$ (the condition in the proposition), we must have $\sigma(i'+k) \in A_3$.
But by Proposition \ref{P:CRCR}, we have $\sigma\kappa'(i'+k) \notin A_3$, and so $\kappa'(i'+k)\neq i'+k$.  Thus, by Lemma \ref{L:easy}(2), $i'+k \in X^{\kappa'}_{12}$ and $\sigma(i') = \sigma\kappa'(i'+k) \in \sigma (X^{\kappa'}_{21})$.
Furthermore, since $\sigma(i') \in A_2$, we have $\sigma(i') \notin X^{\sigma\kappa'}_{21} \subseteq A_1$, and the proof of the claim is complete.

Since $\rho(X^{\kappa}_{21})=X^{\sigma\kappa'}_{21}\subseteq \sigma(X^{\kappa'}_{21})$, we conclude from the above that $|X^{\kappa'}_{21}| - |X^{\kappa}_{21}| = |\sigma([1,\ell])\cap A_2|$.  Let $\varepsilon_{\sigma} = (-1)^{|\sigma([1,\ell])\cap A_2|}$.  Then, by Lemma \ref{L:easy}(2), $$\sgn(\kappa\kappa') = (-1)^{|X^{\kappa'}_{21}| - |X^{\kappa}_{21}|} = (-1)^{|\sigma([1,\ell])\cap A_2|} = \varepsilon_{\sigma}.$$

For general $\sigma$, let $\kappa'' = \prod_{j \in A_2 \cap \sigma^{-1}(A_1)} (j-k, j) \in C_{\s}$, and let $\sigma_0 = \sigma\kappa''$.  Then for all $j \in A_2$, if $\sigma(j) \notin A_1$, then $\sigma_0(j) = \sigma(j) \notin A_1$, while if $\sigma(j) \in A_1$, then $\sigma_0 (j) = \sigma\kappa''(j) = \sigma(j-k) \notin A_1$ since $(\sigma(j-k),\sigma(j)) \notin R_{\s \times \t}$.  Thus $\sigma_0(A_2) \cap A_1 = \varnothing$.  Furthermore, $\sigma_0 \in R_{\s\times \t}C_{\s}R_{\add{\s}{\t}}C_{\s}$ such that $(\sigma_0(j-k), \sigma_0(j)) \notin R_{\add{\s}{\t}}$ for all $j \in A_2$.  If $(\rho, \kappa, \rho',\kappa') \in \m^{-1}(\{\sigma\})$, then $(\rho, \kappa, \rho',\kappa'\kappa'') \in \m^{-1}(\{\sigma_0\})$, and so $\sgn(\kappa\kappa'\kappa'') = \varepsilon_{\sigma_0}$, and hence $\sgn(\kappa\kappa') = \sgn(\kappa'')\varepsilon_{\sigma_0} =: \varepsilon_{\sigma}$, and the proof is complete.
\end{proof}

We will need the next result in the proof of Proposition \ref{P:a_sigma} later.

\begin{lem} \label{L:sum}
Let $r,s \in \mathbb{Z}_{\geq 0}$ with $r+s \leq k$.  Then
$$
\sum_{i=0}^{k-r-s} \binom{k-r-s}{i}(i+s)!(k-i-s)! = \frac{(k+1)!r!s!}{(r+s+1)!}.
$$
\end{lem}

\begin{proof}
We prove by induction on $k-r-s$, where the base case of $k-r-s =0$ can be easily verified.  Assume therefore that $k - r-s>0$.  Using the identity $\binom{n}{a} = \binom{n-1}{a} + \binom{n-1}{a-1}$, together with the convention that $\binom{n}{a} = 0$ if $a > n$ or $a < 0$, we have
{\allowdisplaybreaks
\begin{align*}
\text{LHS} &= \sum_{i=0}^{k-r-s} \left(\binom{k-r-s-1}{i} + \binom{k-r-s-1}{i-1}\right)(i+s)!(k-i-s)! \\
&= \sum_{i=0}^{k-r-s-1} \binom{k-r-s-1}{i} (i+s)!(k-i-s)! \\
&\qquad \quad + \sum_{i=1}^{k-r-s} \binom{k-r-s-1}{i-1}(i+s)!(k-i-s)! \\
&= \sum_{i=0}^{k-r-s-1} \binom{k-r-s-1}{i} (i+s)!(k-i-s)! \\
&\qquad \quad + \sum_{i=0}^{k-r-s-1} \binom{k-r-s-1}{i}(i+s+1)!(k-i-s-1)! \\
&= \frac{(k+1)!(r+1)!s!}{(r+s+2)!} + \frac{(k+1)!r!(s+1)!}{(r+s+2)!} \\
&= \frac{(k+1)!r!s!}{(r+s+2)!}(r+1 + s+1) = \frac{(k+1)!r!s!}{(r+s+1)!},
\end{align*}}
where the fourth equality follows from induction hypothesis.
\end{proof}

\begin{prop} \label{P:a_sigma}
Let $\sigma \in R_{\s\times \t}C_{\s}R_{\add{\s}{\t}}C_{\s}$ such that $(\sigma(j-k), \sigma(j)) \notin R_{\s\times \t}$ for all $j \in A_2$.  Let
\begin{align*}
r &= |\{ i \in [\ell+1,k] \cup A_3 \mid \sigma(i) \in A_1 \}|; \\
s &= |\{j \in A_2 \mid \sigma(j),\sigma(j-k) \notin A_2 \}|.
\end{align*}
  Then
\begin{enumerate}
\item $r \geq k-\ell$ and $r+s \leq \min(k, k-\ell+m)$;
\item
$$
a_{\sigma} 
= \varepsilon_{\sigma} \frac{\ell! m! (k+1)!r!s!}{(r+s+1)!} .
$$
\end{enumerate}
\end{prop}

\begin{proof}
For each $i \in [1,3]$, let $B_i = \{ a \in [1,\ell] \cup A_2 \mid \sigma(a) \in A_i \}$.  Then our imposed condition on $\sigma$ implies that for each $j \in A_2$, at most one of $\sigma(j)$ and $\sigma(j-k)$ may lie in $A_i$, so that $|B_i| \leq \ell$.  Furthermore, for each of the $s$ $j$'s in $A_2$ for which $\sigma(j),\sigma(j-k) \notin A_2$, exactly one of $\sigma(j)$ and $\sigma(j-k)$ lies in $A_1$, while the other lies in $A_3$.  Thus, there are exactly $(|B_1|-s)$ $j$'s in $A_2$ such that exactly one of $\sigma(j)$ and $\sigma(j-k)$ lies in $A_1$ while the other lies in $A_2$.
\begin{enumerate}
  \item There are exactly $k$ $i$'s from $[1,k+\ell+m]$ such that $\sigma(i) \in A_1$, and thus $r+s \leq k$. Exactly $|B_1|$ of these $i$'s lie in $[1,\ell] \cup A_2$, while exactly $r$ of these $i$'s lie outside $[1,\ell] \cup A_2$.  Thus $r+ |B_1| = k$, giving $r = k-|B_1| \geq k - \ell$.  There are exactly $|B_1|$ $j$'s in $A_2$ for which $\{ \sigma(j),\sigma(j-k) \} \cap A_1 \ne \varnothing $, and hence exactly $\ell-|B_1|$ $j$'s in $A_2$ for which $\sigma(j),
      \sigma(j-k) \notin A_1$, which is equivalent to having exactly one of $\sigma(j)$ and $\sigma(j-k)$ lying in $A_2$ while the other lying in $A_3$.  Consequently, $m \geq |B_3| \geq s + (\ell - |B_1|)  = s+\ell - (k-r)$, as desired.

  \item
By Proposition \ref{P:equalsign}, we have
$
a_{\sigma} = \varepsilon_{\sigma} |\m^{-1}(\{\sigma\})|.
$
Recall the paragraph right after Corollary \ref{C:CRCR}) that $(\rho,\kappa,\rho',\kappa') \in \m^{-1}(\{\sigma\})$ if and only if $\kappa'=\kappa_{\sigma,I} = \prod_{j \in I \cup J_{\sigma}'} (j-k,j)$ for some $I \subseteq J_{\sigma}$, where $J_{\sigma} = \{ j \in A_2 \mid \sigma(j), \sigma(j-k) \notin A_3\}$ and $J_{\sigma}' = \{ j \in A_2 \mid \sigma(j) \in A_3\}$, and that
$$
|\{ (\rho,\kappa,\rho') \in R_{\s\times \t}\times C_{\s}\times R_{\add{\s}{\t}} \mid \rho \kappa\rho'\kappa_{\sigma,I} = \sigma \}| = |X^{\sigma \kappa_{\sigma,I}}_{21}|!(k - |X^{\sigma \kappa_{\sigma,I}}_{21}|)! \ell! m!$$
by Proposition \ref{P:RCR}.  We conclude from the above discussion that $|J_{\sigma}| = |B_1|-s = k-r -s $, and that $|X^{\sigma \kappa_{\sigma,I}}_{21}| = s + |\{ j \in J_{\sigma} \mid \sigma\kappa_{\sigma,I} (j) \in A_1 \}|$.  Thus, for each $i \in [0,k-r-s]$, there are exactly $\binom{k-r-s}{i}$ subsets $I$ of $J_{\sigma}$ such that $|X^{\sigma \kappa_{\sigma,I}}_{21}| = s+i$.  Hence,
$$
|\m^{-1}(\{\sigma\})| = \sum_{i=0}^{k-r-s} \binom{k-r-s}{i} (i+s)!(k-i-s)!\ell!m!
= \frac{(k+1)!r!s!}{(r+s+1)!}\ell!m!
$$
by Lemma \ref{L:sum}.
\end{enumerate}
\end{proof}

\begin{lem} \label{L:eg*}
Let $r, s\in \mathbb{Z}_{\geq 0}$ such that $r \geq k-\ell$ and $r+s \leq \min(k,k-\ell+m)$.  Then there exists $\sigma \in R_{\s \times \t}C_{\s} R_{\add{\s}{\t}} C_{\s}$ such that $(\sigma(j-k), \sigma(j)) \notin R_{\s\times \t}$ for all $j \in A_2$, and
$$
r = |\{ i\in [\ell+1,k] \cup A_3 \mid \sigma(i) \in A_1 \}| \quad \text{and} \quad s = |\{ j \in A_2 \mid \sigma(j), \sigma(j-k) \notin A_2 \}|.
$$
\end{lem}

\begin{proof}
Let $\rho' = \prod_{j=1}^{r+s-k+\ell} (j, k+\ell+j) \in R_{\add{\s}{\t}}$ and $\kappa = \prod_{j=1}^s (j, j+k) \in C_{\s}$.  Then for $\sigma = \kappa\rho'$, we have
\begin{align*}
\{ i\in [\ell+1,k] \cup A_3 \mid \sigma(i) \in A_1 \} &= [\ell+1,k] \cup [k+\ell+s+1,2\ell+s+r];\\
\{ j \in A_2 \mid \sigma(j), \sigma(j-k) \notin A_2 \} &= [k+1,k+s].
\end{align*}
\end{proof}

We need the following number theoretic result to express $\theta_{(k,\ell),(m)}$ in a nice closed form.

\begin{lem} \label{L:lcm}
Let $a,b \in \mathbb{Z}_{\geq 0}$. 
\begin{enumerate}

\item We have $$\tfrac{(a+b+1)!}{a!b!} \mid \lcm_{\ZZ}([a+1,a+b+1]).$$

\item The following two subsets of positive integers have the same least common multiple:
\begin{gather*}
\left\{ \tfrac{(r+s+1)!}{r!s!} \mid r, s \in \mathbb{Z}_{\geq 0},\ r \geq a,\ r+s \leq a+b \right\}\\
[a+1,a+b+1].
\end{gather*}
\end{enumerate}
\end{lem}

\begin{proof} \hfill
\begin{enumerate}
\item It suffices to show that $v_p(\tfrac{(a+b+1)!}{a!b!}) \leq v_p(\lcm_{\ZZ}([a+1,a+b+1]))$ for any prime integer $p$. We have $v_p(\tfrac{(a+b+1)!}{a!b!}) = v_p((a+1) \binom{a+b+1}{a+1}) = v_p(a+1) + v_p(\binom{a+b+1}{a+1})$.
    Let
$$
a+1 = \sum_{i=0}^{\infty} \alpha_i p^i \quad \text{and} \quad  b = \sum_{i=0}^{\infty} \beta_i p^i
$$ be the $p$-adic decompositions of $a+1$ and $b$. Let $$I = \{ i \in \ZZ_{\geq 0} \mid \sum_{j=0}^i (\alpha_j + \beta_j) p^j \geq p^{i+1} \}.$$
By a result of Kummer \cite[p.116]{Kummer}, we have $v_p(\binom{a+b+1}{a+1} )= |I|$.

If $I = \varnothing$, then
$$v_p(\tfrac{(a+b+1)!}{a!b!}) = v_p(a+1) \leq v_p(\lcm_{\ZZ}([a+1,a+b+1])).$$
Assume thus $I\ne \varnothing$, and let $i_0 = \min(I)$ and $i_1 = \max(I)$.  Since  $\alpha_i = 0$ for all $j < v_p(a+1)$, we have $\sum_{j=0}^i (\alpha_j + \beta_j) p^j = \sum_{j=0}^i  \beta_j p^j < p^{i+1}$ for all $i < v_p(a+1)$, so that $i_0 \geq v_p(a+1)$.  Let $c = p^{i_1+1} + \sum_{i = i_1+1}^{\infty} \alpha_i p^i$.  Then $c -(a+1) = p - \sum_{i =0}^{i_1} \alpha_i p^i >0$, while
$$
(a+b+1) - c = \sum_{i=i_1+1}^{\infty} \beta_i + \sum_{j=0}^{i_1} (\alpha_j+\beta_j)p^j - p^{i_1+1} \geq 0$$
since $i_1 \in I$, so that $c \in [a+1,a+b+1]$.  Thus,
\begin{align*}
v_p(\tfrac{(a+b+1)!}{a!b!}) = v_p(a+1) + v_p(\tbinom{a+b+1}{a+1})
&\leq i_0 + |I| \\ &\leq i_1 + 1  \leq v_p(c) \leq v_p(\lcm_{\ZZ}([a+1,a+b+1]),
\end{align*}
and we are done.

\item 

    Fix $a\in\ZZ_{\geq 0}$. Let $S_b=\left\{ \tfrac{(r+s+1)!}{r!s!} \mid r, s \in \mathbb{Z}_{\geq 0},\ r \geq a,\ r+s \leq a+b \right\}$, and $\ell_b = \lcm_{\mathbb{Z}} (S_b)$. We prove by induction on $b\in\ZZ_{\geq 0}$, with the base case of $b=0$ being trivial. Assume thus that $b>0$ and $\ell_{b-1} = \lcm_{\ZZ}([a+1,a+b])$. When $s = 0$ we have $\frac{(r+s+1)!}{r!s!} = r+1$.  Thus $[a+1,a+b+1] \subseteq S_b$, so that $$\lcm_{\mathbb{Z}} ([a+1,a+b+1]) \mid \ell_b.$$
    Conversely, as $S_b =  S_{b-1} \cup \{ \frac{(a+b+1)!}{(a+b-s)!s!} \mid s \in [0,b] \}$, we have, for any $x \in S_{b-1}$, $x \mid \ell_{b-1}  = \lcm_{\ZZ}([a+1,a+b]) \mid \lcm_{\ZZ}([a+1,a+b+1])$ by induction, while $$\tfrac{(a+b+1)!}{(a+b-s)!s!} \mid \lcm_{\ZZ}([a+b-s+1,a+b+1]) \mid \lcm_{\ZZ}([a+1,a+b+1])$$ for all $s \in [0,b]$ by part (1), and the proof is complete.
\end{enumerate}
\end{proof}

\begin{thm} \label{T:main}
We have
\begin{align*}
\theta_{(k,\ell),(m)} = \frac{(k+1)!\, \ell!\, m!}{\lcm_{\mathbb{Z}}([k-\ell+1, k-\ell+1+\min(\ell,m)])}.
\end{align*}
\end{thm}

\begin{proof}
By Lemmas \ref{L:a_sigma=0} and \ref{L:eg*}, and Proposition \ref{P:a_sigma}, we have
\begin{align*}
\theta_{(k,\ell),(m)} &= \gcd\nolimits_{\mathbb{Z}}\{ \tfrac{\ell! m! (k+1)!r!s!}{(r+s+1)!} \mid r,s \in \mathbb{Z}_{\geq 0},\ r \geq k-\ell,\ r+s \leq \min(k, k-\ell+m) \} \\
&= \frac{\ell!m!(k+1)!}{\lcm_{\mathbb{Z}} \{\frac{(r+s+1)!}{r!s!} \mid r,s \in \mathbb{Z}_{\geq 0},\ r \geq k-\ell,\ r+s \leq k-\ell+\min(\ell,m) \}} \\
&= \frac{\ell!m!(k+1)!}{\lcm_{\mathbb{Z}}([k-\ell+1,k-\ell+1 +\min(\ell,m)])},
\end{align*}
where the last equality follows from Lemma \ref{L:lcm}(3).
\end{proof}

\begin{rem} \label{rem:final-rem} \hfill
\begin{enumerate}
\item
Theorem \ref{T:main} in particular provides an alterative proof that statements (i) and (iii) of Proposition \ref{prop:Donkin-criterion} are equivalent, which we now demonstrate.
Putting $\ell=0$ in Theorem \ref{T:main}, we obtain $\theta_{(k),(m)} = k!m!$, and so by Theorem \ref{thm:split-condition-main}, the canonical $\mathrm{GL}_N(\FF)$-morphism $\iota_{(k),(m)}$ splits over $\ZP$ if and only if $p \nmid \frac{\phl^{(k+m)}}{k!m!} = \frac{(k+m)!}{k!m!} = \binom{k+m}{m}$ when $N\geq k+m$.
\newcommand{\PR}[3]{M_{#1}({#2},{#3})}
When $N<k+m$, let $\PR{\ZP}{N}{k+m}$ (respectively, $\PR{\ZP}{k+m}{k+m}$) denote the category of homomogeneous polynomial $\mathrm{GL}_N(\ZP)$-modules (respectively, $\mathrm{GL}_{k+m}(\ZP)$-modules) of degree $k+m$, and consider the `truncation' functor (see, for example, \cite[\S6.5]{Green80})
\[
f: \PR{\ZP}{k+m}{k+m} \to \PR{\ZP}{N}{k+m}.
\]
Let $E$ and $\hat{E}$ be the natural $\mathrm{GL}_N(\ZP)$- and $\mathrm{GL}_{k+m}(\ZP)$-modules respectively. The Weyl module $\Delta_{\ZP}(r)$, for $\mathrm{GL}_N(\ZP)$ and $\mathrm{GL}_{k+m}(\ZP)$, is isomorphic to the divided power $D^r E$ and $D^r \hat{E}$ respectively, for all $r\geq 1$. 
Since $k+m\geq 2$, both $D^k \hat{E}\otimes D^m\hat{E}$ and $D^{k+m}\hat{E}$ are projective modules in $\PR{\ZP}{k+m}{k+m}$, while $f(D^k\hat{E}\otimes D^m\hat{E}) = D^kE\otimes D^m E$ and  $f(D^{k+m}\hat{E}) = D^{k+m}E$ are projective modules in $\PR{\ZP}{N}{k+m}$. It follows that $f$ induces the natural isomorphism
\begin{align*}
&\Hom_{\PR{\ZP}{k+m}{k+m}} (D^k\hat{E}\otimes D^m \hat{E},D^{k+m}\hat{E})\\
\cong\ & \Hom_{\PR{\ZP}{N}{k+m}} (f(D^k\hat{E}\otimes D^m \hat{E}),f(D^{k+m}\hat{E})) \\
=\ &\Hom_{\PR{\ZP}{N}{k+m}}(D^k E\otimes D^mE,D^{k+m}E).
\end{align*}
Thus the canonical morphism $D^{k+m}E \to D^kE\otimes D^mE$ splits over $\ZP$ if and only if $D^{k+m}\hat{E} \to D^k\hat{E}\otimes D^m\hat{E}$ splits over $\ZP$. As a result, $\iota_{(k),(m)}$ splits over $\ZP$ if and only if $p\nmid \binom{k+m}{k}$, irrespective of the value of $N$.

\item
Since $\phl^\nu$ is the product of all hook lengths in the Young diagram $[\nu]$, Corollary \ref{cor:theta-gcd} may suggest to some readers that $\theta_{\lambda,\mu}$ could be a product of some subset of hook lengths in $[\lambda+\mu]$ or in $[\lambda]\cup[\mu]$, and that furthermore it might be possible to describe such a subset combinatorially. The appearance of the least (positive) common multiple of an integer interval in Theorem \ref{T:main} indicates that if this is indeed true, the description of such a subset may be rather complicated.

\item

After this paper has been submitted, the authors study Young's seminormal basis vectors and their denominators, and are able to provide closed formula for $\dd_{\add{\IT^\lambda}{\IT^\mu}}$ directly when $\mu$ is a one-row partition and $\lambda$ is either a two-row partition or a hook partition.  This provides an alternative way of computing $\theta_{\lambda,\mu}$ via Theorem \ref{thm:denominator} for the examples listed in this section.  We refer the interested reader to our forthcoming paper \cite{FLT2}.
\end{enumerate}
\end{rem}


\begin{thebibliography}{99}

\bibitem{ABW} K. Akin, D. Buchsbaum and J. Weyman,
Schur functors and Schur complexes,
\textit{Adv. Math. } \textbf{44} (1982), 207--278.

 \bibitem{Andersen83} H. Andersen,
 Filtrations of cohomology modules for Chevalley groups,
 \textit{Ann. Sci. \'{E}cole Norm. Sup.} \textbf{16} (1983), 495--528.

 \bibitem{Andersen87} H. Andersen,
 Jantzen's filtration of Weyl modules,
 \textit{Math. Z.} \textbf{194} (1987), 127--142.


\bibitem{BK} J. Brundan and A. Kleshchev, On translation functors for general linear and symmetric groups, \textit{Proc. Lond. Math. Soc.} \textbf{80} (2000), 75--106.

 \bibitem{Donkin98} S. Donkin,
 The $q$-Schur Algebra,
 \textit{London Math. Soc. Lecture Note Series}
 \textbf{253}, Cambridge University Press, 1998.


\bibitem{FLT2} M. Fang, K. J. Lim and K. M. Tan,
Young's seminormal basis vectors and their denominators,
in preparation, 2020.

 \bibitem{Fulton97} W. Fulton,
 Young Tableaux, 
 \textit{London Math. Soc. Student Texts} \textbf{35},
 Cambridge University Press, 1997.

\bibitem{GLOW17} E. Giannelli, K. J. Lim, W. O'Donovan and M. Wildon, On signed Young permutation modules and signed $p$-Kostka numbers,
    \textit{J. Group Theory} \textbf{20} (2017), 637--679.

 \bibitem{Green80} J. Green,
 Polynomial Representations of $GL_n$, \textit{Lecture Notes in Mathemtics} \textbf{830},
 Springer-Verlag, New York, 1980.


 \bibitem{HN04} D. J. Hemmer and D. K. Nakano,
 Specht filtration for Hecke algebras of type A,
 \textit{J. London Math. Soc.} \textbf{69} (2004), 623--638.


 \bibitem{James78} G. D. James,
 The Representation Theory of the Symmetric Groups,
 \textit{Lecture Notes in Mathematics}
 \textbf{682}, Springer-Verlag, 1978.


 \bibitem{Jantzen77} J. C. Jantzen,
 Darstellungen halbeinfacher gruppen und kontravariante formen,
 \textit{J.\ Reine Angew.\ Math.} \textbf{290} (1977), 117--141.

 \bibitem{Jantzen03} J. C. Jantzen,
 Representations of Algebraic Groups: Second Edition,
 \textit{Math. Survey and Monographs} \textbf{107},
 American Mathematical Society, 2003.



\bibitem{Kummer} E. E. Kummer, \"Uber die Erg\"anzungss\"atze zu den allgemeinen Reciprocit\"atsgesetzen, {\em J.\ Reine Angew.\ Math.}\ \textbf{44} (1852), 93--146.

 \bibitem{LyleMathas10} S. Lyle and A. Mathas,
 Carter-Payne homomorphisms and Jantzen filtrations,
 \textit{ J. Algebraic Combin.}~\textbf{32} (2010), 417--457.


 \bibitem{Mathasbook99} A. Mathas,
 Iwahori-Hecke Algebras and Schur Algebras of the Symmetric Group,
 \textit{University Lecture Series} \textbf{15}, American Mathematical Society, 1999.


\bibitem{Murphy92} G. E. Murphy,
On the representation theory of the symmetric groups and associated Hecke algebras,
 \textit{J. Algebra}, \textbf{152} (1992), 492--513.


 \bibitem{Raicu14} C. Raicu,
 Products of Young symmetrizers and ideals in the generic tensor algebra, \textit{J. Algebraic Combin.} \textbf{39} (2014), 247--270.

\bibitem{RHansen10} S. Ryom-Hansen, On the denominators of Young's seminormal basis, arXiv:0904.4243v3.

\bibitem{RHansen13} S. Ryom-Hansen, Young's seminormal form and simple modules
for $S_n$ in characteristic $p$, \textit{Algebr. Represent. Theory} \textbf{16}  (2013), 1587--1609.



\bibitem{Shan12} P. Shan, Graded decomposition matrices of $v$-Schur algebras via Jantzen filtration, \textit{Represent. Theory} \textbf{16} (2012), 212--269.

\bibitem{T97} B. Totaro, Projective resolutions of representations of $GL(n)$, \textit{J.\ Reine Angew.\ Math.} \textbf{482} (1997), 1--13.

\bibitem{Xi96} N. Xi, Irreducible modules of quantized enveloping algebras at roots of $1$, \textit{Publ. RIMS. Kyoto Univ.} \textbf{32} (1996), 235--276.

\end{thebibliography}
\end{document}